\documentclass[11pt,reqno,final]{amsart}
\usepackage{subeqnarray}
\usepackage{cases}
\usepackage{stmaryrd}
\usepackage{color}
\usepackage{amsfonts}
\usepackage{mathrsfs}
\usepackage{amssymb,amsmath,amsthm}
\usepackage{epsfig}
\usepackage{graphicx}
\usepackage{appendix}
\usepackage{lineno}
\usepackage{enumerate}
\usepackage{comment}
\usepackage{caption}
\usepackage[notcite,notref]{showkeys}
\usepackage[numbers,sort&compress]{natbib}
\usepackage{empheq}
\usepackage{fullpage}
\newcommand{\dps}{\displaystyle}
\newcommand{\tps}{\textstyle}
\newtheorem{theorem}{\indent Theorem}[section]
\newtheorem{lemma}{\indent Lemma}[section]

\newtheorem{remark}{\indent Remark}[section]

\newcommand{\ba}{\begin{array}}\newcommand{\ea}{\end{array}}
\newcommand{\be}{\begin{eqnarray}}\newcommand{\ee}{\end{eqnarray}}
\newcommand{\beq}{\begin{equation*}}\newcommand{\eeq}{\end{equation*}}
\newcommand{\bex}{\begin{eqnarray*}}
\newcommand{\eex}{\end{eqnarray*}}

\def\bq{\begin{equation}}
\def\eq{\end{equation}}
\def\beq{\begin{equation*}}
\def\eeq{\end{equation*}}
\def\br{\begin{eqnarray}}
\def\er{\end{eqnarray}}
\def\brr{\bq\begin{array}{r@{}l}}
\def\err{\end{array}\eq}
\def\bry{\beq\begin{array}{r@{}l}}
\def\ery{\end{array}\eeq}
\def\brl{\bq\begin{array}{l}}
\def\erl{\end{array}\eq}
\def\bryl{\beq\begin{array}{l}}
\def\eryl{\end{array}\eeq}

\font\tenbi=cmmib10   at 11 pt
\font\sevenbi=cmmib10 at 9pt
\font\fivebi=cmmib7 at 6pt
\newfam\bifam
\textfont\bifam=\tenbi \scriptfont\bifam=\sevenbi  \scriptscriptfont\bifam=\fivebi

\def\bi{\fam\bifam\tenbi}
\font\sixtdb=msbm10 at 16 pt \font\tendb=msbm10 at 12 pt  \font\sevendb=msbm7
\newfam\dbfam
\textfont\dbfam=\sixtdb

\textfont\dbfam=\tendb \scriptfont\dbfam=\sevendb

\def\Dt {\triangle t}

\def\n{{\bi n}}
\def\x{{\bi x}}
\def\vPhi{{\vec{\Phi}}}
\def\ve{{\vec{e}\ \!}}

\def\R{\mathbb{R}}

\def\Dt {\tau}

\title[A  linear doubly stabilized Crank-Nicolson scheme for the Allen-Cahn equation]
{A  linear doubly stabilized  Crank-Nicolson scheme for the Allen-Cahn equation with a general mobility$^*$}
\author[Dianming Hou, Lili Ju and Zhonghua Qiao]
{Dianming Hou$^{1}$
\quad
Lili Ju$^{2}$
\quad
Zhonghua Qiao$^{3}$
}
\thanks{
$^{1}$School of Mathematics and Statistics, Jiangsu Normal University, Xuzhou, Jiangsu 221116, China. Email: {\tt dmhou@stu.xmu.edu.cn}. Current address: Department of Applied Mathematics, The Hong Kong Polytechnic University, Hung Hom, Kowloon, Hong Kong. D. Hou's work is partially supported by Natural Science Foundation of China grant 12001248,  Jiangsu Province Higher Education Institutions grant
BK20201020,  Jiangsu Province Universities  Science Foundation grant 20KJB110013 and  Hong Kong Polytechnic University grant 1-W00D. \\
$^{2}$Department of Mathematics, University of South Carolina, Columbia, SC 29208, USA. Email:  {\tt ju@math.sc.edu}. L. Ju's work is partially supported by US National Science Foundation grant DMS-2109633.\\
$^{3}$Department of Applied Mathematics, The Hong Kong Polytechnic University, Hung Hom, Kowloon, Hong Kong. Email:  {\tt zqiao@polyu.edu.hk}. Z. Qiao‘s work   is partially supported by Hong Kong Research Council RFS grant RFS2021-5S03 and GRF grant 15302919, Hong Kong Polytechnic University grant 4-ZZLS, and CAS AMSS-PolyU Joint Laboratory of Applied Mathematics.
}
\begin{document}
\graphicspath{{figures/},}
\maketitle

\begin{abstract}
In this paper, a linear second order numerical scheme is developed and investigated for the Allen-Cahn equation with a general positive mobility.
In particular, our fully discrete scheme is mainly constructed based on the Crank-Nicolson formula for temporal discretization
and the central finite difference method for spatial approximation, and two extra stabilizing terms are also  introduced for the purpose of improving numerical stability.  The proposed scheme is shown to unconditionally preserve the maximum bound principle (MBP) under  mild restrictions on the stabilization parameters, which is of practical importance for achieving good accuracy and stability simultaneously. With the help of uniform boundedness of the numerical solutions due to MBP, we then successfully derive $H^{1}$-norm and $L^{\infty}$-norm error estimates  for the Allen-Cahn equation with a constant and a variable mobility, respectively.
Moreover, the energy stability of the proposed scheme is also obtained in the sense that the discrete free energy is uniformly bounded by the one at the initial time plus a {\color{black}constant}. Finally, some numerical experiments are carried out to verify the theoretical results and illustrate the
performance of the proposed scheme with a time adaptive strategy.
\end{abstract}

\section{Introduction}
\setcounter{equation}{0}

In this paper, we study numerical solution of the following Allen-Cahn equation with a general  mobility $M(\phi)\geq M_{0}>0$:
\brr\label{prob}
\begin{cases}
\dps\frac{\partial \phi}{\partial t}=\dps-M(\phi)\big(-\varepsilon^{2}\Delta\phi+F'(\phi)\big),&\quad (\x,t)\in\Omega\times (0,T],\\
\phi(\x,0)=\phi_{0}(\x),&\quad \x\in\Omega,
\end{cases}
\err
which often arises from modeling of phase transitions and interfacial dynamics in materials science.
 Here, $\Omega$ is a  bounded Lipschitz domain in $\R^{d}$ $(d=1,2,3)$,
 $T>0$ is the terminal time,
 $\phi(\x,t)$ is the unknown phase function, the positive parameter $\varepsilon$ is called the diffuse interface width parameter,
 and $F(\phi)=\frac14(1-\phi^{2})^{2}$ is  the double-well potential function. We also assume that the problem is subject to suitable boundary conditions such as  the homogeneous Neumann, the periodic, or the homogeneous Dirichlet boundary condition.
The Allen–Cahn equation \eqref{prob} can be viewed as the $L^{2}$ gradient flow of the energy
\bq\label{energy}
E(\phi)=\int_{\Omega}\Big(\frac{\varepsilon^{2}}{2}|\nabla\phi|^{2}+F(\phi)\Big)d\x,
\eq
which leads to the dissipation of the free energy $E(\phi)$ over time, that is
 \bq\label{EDlaw}
 \frac{d}{dt}E(\phi)
 =-\int_{\Omega}M(\phi)\mu^{2}d\x\leq0.
 \eq
Another intrinsic property of the Allen–Cahn equation \eqref{prob} is the maximum bound principle (MBP), i.e.,  if $|\phi(\x,0)|\leq1$  for all $\x\in\Omega$ then $|\phi(\x,t)|\leq1$ for all $\x\in\Omega$ and $t\geq 0$, and one can refer to \cite{STY16} for more discussions.
To numerically investigate the Allen-Cahn equation \eqref{prob}, it is essentially important for the numerical schemes to preserve these physical properties in the discrete level, particularly  the preservation of MBP, otherwise it could encounter the negativity of the mobility $M(\phi)$ which may lead to failing of the numerical schemes.

Over the past few decades, a great deal of works \cite{FSTY13,TY16,STY16,YYZ22,XHF20} has been devoted to developing  structure-preserving time-stepping schemes for the Allen-Cahn equation, particularly for the models with constant mobility.
Among the existing works, first order (in time) liner stabilized semi-implicit schemes combined with the central finite difference method for spatial discretization were  proposed for the Allen-Cahn equation \eqref{prob} in \cite{TY16} and the generalized case with a advection term in \cite{STY16}. These proposed schemes unconditionally preserve the discrete MBP in both cases and the energy stability in the constant mobility case.
 A nonlinear second-order Crank-Nicolson scheme for the space-fractional Allen-Cahn equation was developed  in \cite{HTY17}, in which the convex splitting approach was taken to deal with the nonlinear term. This scheme was proved to conditionally preserve the discrete MBP and the discrete energy dissipation law, and  some corresponding error estimates were also obtained.
A nonlinear two-step second-order backward differentiation formula (BDF2) scheme with nonuniform grids for the Allen-Cahn equation was studied in \cite{LTZ20}, in which the nonlinear term was treated fully implicitly.
The MBP preservation and energy stability of the developed scheme were obtained under some constraints on the time step size and the ratio of two successive time steps.
Recently, Hou et al. \cite{HL20} proposed a linear stabilized second-order Crank-Nicolson/Adams-Bashforth scheme for the Allen-Cahn equation.
It was shown that the numerical scheme preserved the discrete MBP and a modified energy stability conditionally.
Very recently, a linear stabilized BDF2 scheme with variable time steps was numerically studied for the Allen-Cahn equation \eqref{prob} in \cite{HJQ22}, and the discrete MBP of the developed variable-step scheme has been rigorously obtained with certain constrains on the time-step sizes and the adjacent time step ratios.
We would like to remark  that there is no linear second-order unconditional MBP preservation scheme among the above existing works.

A series of structure-preserving exponential time differencing (ETD) and integrating factor Runge-Kutta (IFRK) methods were also investigated for a class of semilinear parabolic equations  in \cite{JLQY21,LLJF21,JJLL21,LJCF21,HJL22,CJLL23,DJLQ19,Liu2022B177,NS22}, all of them unconditionally and conditionally preserve the discrete MBP.
In recent work \cite{DJLQ21}, Du et al. established an abstract framework of MBP investigation for problem \eqref{prob}, where sufficient conditions on linear and nonlinear operators are given such that the equation satisfies MBP and the corresponding MBP preserving first-order ETD and second-order ETD Runge-Kutta (ETDRK) schemes were developed and analyzed.
It was proved in \cite{FY22} that the stabilized first and second order ETDRK schemes unconditionally preserve the discrete energy dissipation law for the Allen-Cahn equation.
 By combining the scalar auxiliary variable (SAV) approach with linear stabilized ETD methods,  some novel SAV-EI schemes for the Allen-Cahn equation  were proposed \cite{JLQ22_1,JLQ22_2} which satisfy both the energy dissipation law and MBP in the discrete level.
 Several third- and fourth-order MBP-preserving schemes \cite{LLJF21,ZYQGS21,ZYQS21,ZYQCS22,CQS23} were developed and analyzed for the Allen–Cahn
equation using the integrating factor Runge–Kutta approach.
 An arbitrarily high-order multistep exponential integrator method
was given in \cite{LYZ20} by enforcing the maximum bound via a cut-off operation.
Due to that fact that these high-order MBP-preserving methods are derived from either the variation-of-constant formula or an exponential transformation of the solution, it seems not easy to extend these approaches to the Allen-Cahn equation \eqref{prob} with a general variable mobility.

The goal of this paper is to propose and analyze a linear second-order, unconditionally MBP preserving scheme for the Allen-Cahn equation  \eqref{prob} with a general mobility, based on  the Crank-Nicolson time-stepping formula and the linear stabilizing approach.
The novelties and significance of this paper include:
first,  a  linear doubly stabilized Crank-Nicolson scheme for the model is constructed for the first time, which is of second order accuracy and possesses the property of unconditional MBP preservation; second,  energy stability of the proposed scheme is established in the sense  that the discrete energy at all time steps  is uniformly bounded by the initial one; third,  error estimates for the proposed scheme with nonuniform temporal mesh are successfully established in the $H^{1}$-norm for the case of constant mobility and in the $L^{\infty}$-norm for the case with variable mobility, respectively;
fourth, the proposed scheme is very efficient (there are only two Poisson-type equations to be solved at each time step) and  can be easily adopted with existing time adaptive strategies.

The rest of the paper is organized as follows. In Section 2 , we present the fully-discrete linear doubly stabilized Crank-Nicolson scheme for the Allen-Cahn equation with a general mobility \eqref{prob} and prove its unconditional preservation of discrete MBP. Some fully-discrete error estimates in the $L^{\infty}$ and $H^{1}$ norms and energy stability are then derived for the propose scheme in Section 3. In Section 4,  various numerical experiments are presented to verify the theoretical results and demonstrate the performance of the proposed scheme. Finally,  concluding remarks are drawn in Section 5.

\section{The fully-discrete  linear doubly stabilized Crank-Nicolson scheme}
\setcounter{equation}{0}
Without loss of generality, the two-dimensional problem ($d=2$) with the homogenous Neumann boundary condition, i.e., $\frac{\partial\phi}{\partial \n}\big|_{\partial\Omega}=0 $ is considered in what follows. We also note that it is straightforward to extend the proposed scheme and  corresponding analysis results to the cases of higher dimensional spaces and/or  other boundary conditions.

\subsection{Spatial discretization by central difference}\label{sub1}

We use the notations and preliminary results of the central difference function spaces and operators reported in \cite{WISE10,BLWW13,BHLWWZ13,SWWW12,HWWL09,WWL09,LSR19,WW88}. For more complete details, one can refer to these works.
For simplicity, we consider a square computational domain $\Omega=(0,L)\times(0,L),$ and the uniform spatial grid spacing $h=L/M$.
Define the following discrete function spaces:
\bry
\mathcal{C}_{h}=\;&\dps\{U: \mathbf{C}_{M}\times\mathbf{C}_{M}\rightarrow\R\;\big|\;U_{i,j}, \;1\leq i,j\leq M\},\\[4pt]
 e^{x}_{h}=\;&\dps\{U: \mathbf{E}_{M}\times\mathbf{C}_{M}\rightarrow\R\;\big|\;U_{i+\frac{1}{2},j}, \;0\leq i\leq M,\;1\leq j\leq M\},\\[4pt]
 e^{y}_{h}=\;&\dps\{U: \mathbf{C}_{M}\times\mathbf{E}_{M}\rightarrow\R\;\big|\;U_{i,j+\frac{1}{2}}, \;1\leq i\leq M,\;0\leq j\leq M\},\\[4pt]
 e^{x}_{0,h}=\;&\dps\{U\in e^{x}_{h}\;\big|\;U_{\frac{1}{2},j}=U_{M+\frac{1}{2},j}=0,\; 1\leq j\leq M\},\\[4pt]
 e^{y}_{0,h}=\;&\dps\{U\in  e^{y}_{h}\;\big|\;U_{i,\frac{1}{2}}=U_{i,M+\frac{1}{2}}=0,\; 1\leq i\leq M\},
\ery
where the two types of point sets $E_{M}$ and $C_{M}$ are given by
$$\mathbf{E}_{M}=\{x_{i+\frac{1}{2}}=ih\;\big|\; i=0,1, \cdots,M\},\qquad \mathbf{C}_{M}=\{x_{i}=\big(i-\textstyle\frac{1}{2}\big)h\;\big|\; i=1, \cdots,M\}.$$
Then, we define the discrete gradient operator $\nabla_{h} =(\nabla^x_{h}, \nabla^y_{h}): \mathcal{C}_h\rightarrow( e_{0,h}^{x},  e^{y}_{0,h})$ by
\bq
(\nabla^x_{h}U)_{i+\frac{1}{2},j}=\dps\frac{U_{i+1,j}-U_{i, j}}{h}, \quad 1\leq i\leq M-1,\;1\leq j\leq M,
\eq
\bq
(\nabla^y_{h}U)_{i,j+\frac{1}{2}}=\dps\frac{U_{i,j+1}-U_{i, j}}{h}, \quad 1\leq i\leq M,\;1\leq j\leq M-1,
\eq
for any $U\in \mathcal{C}_h$,
and the discrete divergence operator $\nabla_{h}\cdot:( e_{h}^{x},  e^{y}_{h})\rightarrow\mathcal{C}_h$ by
\bq
(\nabla_{h}\cdot (U^{x},U^{y})^T)_{i,j}=\tps\frac{U^x_{i+\frac{1}{2},j}-U^x_{i-\frac{1}{2},j}}{h}+\tps\frac{U^y_{i,j+\frac{1}{2}}-U^y_{i,j-\frac{1}{2}}}{h},\quad 1\leq i,j\leq M
\eq
for any $(U^{x},U^{y})^{T}\in ( e_{h}^{x},  e^{y}_{h}).$ Note that the above discrete gradient and divergence operators  are compatible with the homogeneous Neumann boundary condition.
Then, we use the discrete gradient and divergence operators to obtain the  discrete Laplacian $\Delta_{h}:\mathcal{C}_{h}\rightarrow\mathcal{C}_{h}$, given by
\bq
(\Delta_{h}U)_{i,j}=(\nabla_{h}\cdot(\nabla_{h} U))_{i,j},\quad 1\leq i,j\leq M.
\eq
Next we are ready to define the following discrete inner-products:
\bry
&\big<U,V\big>_{\Omega}=\dps h^{2}\sum_{i,j=1}^{M}U_{i,j}V_{i,j},\quad \forall\,U,V\in\mathcal{C}_{h},\\[7pt]
&[U^x,V^x]_{x}=\big<a_{x}(U^xV^x),1\big>_{\Omega},\quad \forall\,U^x,V^x\in e^{x}_{h},\\[7pt]
&[U^y,V^y]_{y}=\dps\big<a_{y}(U^yV^y),1\big>_{\Omega},\quad \forall\,U^y,V^y\in e^{y}_{h},\\[7pt]
& [(U^{x},U^{y})^{T},(V^{x},V^{y})^{T}]_{\Omega}=[U^{x},V^{x}]_{x}+[U^{y},V^{y}]_{y},\\
\ery
where $a_{x}: e^{x}_{h}\rightarrow\mathcal{C}_h$ and $a_{y}: e^{y}_{h}\rightarrow\mathcal{C}_h$ are the two average operators defined by $(a_{x}U)_{i,j}=({U_{i+1/2,j}+U_{i-1/2,j}})/2$
 and $(a_{y}U)_{i,j}=({U_{i,j+1/2}+U_{i,j-1/2}})/2$ for $1\leq i,j\leq M$, respectively.
Then,  for any $U\in\mathcal{C}_{h}$, its corresponding discrete  $L^2, H^1$ semi-norms and norms, and the $L^{\infty}$-norm are respectively given by:
\bry
&\dps\|U\|^{2}_{h}=\big<U,U\big>_{\Omega},\quad \|\nabla_{h}U\|^{2}_{h}=[\nabla_{h}U,\nabla_{h}U]_{\Omega}=[d_{x}U,d_{x}U]_{x}+[d_{y}U,d_{y}U]_{y},\\[7pt]
&\dps\|U\|^{2}_{H^{1}_{h}}=\|U\|^{2}_{h}+\|\nabla_{h}U\|^{2}_{h},\quad \|U\|_{\infty}=\max_{0\leq i\leq N}\sum_{j=0}^{N}|U_{i,j}|.
\ery
From these above definitions, we obtain the following results.
\begin{lemma}[\cite{LSR19,WW88}]\label{intpat}
For any $U,V\in\mathcal{C}_{h}$, it holds
\bq
-\big<\Delta_{h}U,V\big>_{\Omega}=[\nabla_{h}U,\nabla_{h}V]_{\Omega}.
\eq
\end{lemma}

\subsection{Time integration by linear Crank-Nicolson scheme}
Let $0=t_{0}<t_1<t_2<\cdots<t_N=T$ be a general partition of the time interval $[0,T]$ with time step size $\Dt_{n}=t_{n}-t_{n-1}$ for $n=1,2,\cdots,N$.
We denote the maximum time step size of such time partition by $\tau= \max_{1\leq n\leq N}\Dt_{n}$, and the operator pointwisely  limiting a function onto $\mathcal{C}_h$ by $\Pi_{\mathcal{C}_h}$.
Let $\vec{U}=[U_{1,1},\cdots,U_{1,M};\cdots;U_{M,1},\cdots,U_{M,M}]^{T}\in{\mathbb R}^{M^2}$ be vector form of $U\in\mathcal{C}_{h}$.

The fully-discrete linear stabilized first-order BDF scheme for solving the Allen-Cahn equation with general mobility \eqref{prob} reads as follows, seeing also \cite{STY16,TY16}: given $\Phi^0=\Pi_{\mathcal{C}_h}\phi_0$,  for $n=0,1,\cdots,N-1$, find $\Phi^{n+1}\in\mathcal{C}_{h}$ such that
 \bq\label{BDF_1}
\frac{\Phi^{n+1}-\Phi^{n}}{\Dt_{n+1}}-\varepsilon^{2}M(\Phi^{n})\Delta_{h}\Phi^{n+1}+f(\Phi^{n})+S_{1}(\Phi^{n+1}-\Phi^{n})=0,
  \eq
  where $f(\phi) =M(\phi)F'(\phi)$ and  $S_{1}$ is  a nonnegative stabilizing parameter. Hereafter, we call the above scheme BDF1 and  denote it as $\Phi^{n+1} = {\rm BDF1}(\Phi^n,\tau_{n+1})$.
  Moreover, it also can be rewritten in vector form as follows:
  \bq\label{BDF1_tsor}
\frac{\vPhi^{n+1}-\vPhi^{n}}{\Dt_{n+1}}-\varepsilon^{2}\Lambda^{n}D_{h}\vPhi^{n+1}+f(\vPhi^{n})+S_{1}(\vPhi^{n+1}-\vPhi^{n})=0,
  \eq
   where
$D_{h}=I\otimes G_{h}+G_{h}\otimes I\in {\mathbb R}^{M^2\times M^2}$. Here, $I$ denotes the identity matrix (with the matched dimensions) and $G_{h}$ is a diagonally dominant tridiagonal Matrix, given by
\[ G_{h}=\frac{1}{h^{2}}
\begin{pmatrix}
-1&1&&&&\\
1&-2&1&&&\\
&\ddots &\ddots &\ddots&\\
&&1&-2&1&\\
&&&1&-1&\\
\end{pmatrix}_{M\times M}.\]
The matrix $f(\vPhi^{n})$ is defined elementwise, that is  $f(\vPhi^{n})=\Lambda^{n}\big(\big(\vPhi^{n}\big)^{.3}+\vPhi^{n}\big)$ with a diagonal matrix $\Lambda^{n}=\mbox{diag}(M(\vPhi^{n}))$.

From the definition of the free energy $E(\phi)$ in \eqref{energy}, we define an analogous discrete energy $E_{h}(\Phi^{n})$ in the form of
\brr\label{dis_eg}
E_{h}(\Phi^{n})&\dps=\frac{\varepsilon^{2}}{2}[\nabla_{h}\Phi^n,\nabla_{h}\Phi^n]_{\Omega}+\big<F(\Phi^{n}),1\big>_{\Omega}\\
&\dps=-\frac{h^{2}\varepsilon^{2}}{2}(\vPhi^{n})^{T}D_{h}\vPhi^{n}+h^{2}\sum_{i=1}^{M^{2}}F(\vPhi^{n}_{i}).
\err

As reported in Theorem 3.2 in \cite{STY16} and Theorem 3 in \cite{TY16}, the fully-discrete BDF1 scheme \eqref{BDF_1} is unconditionally energy stable and MBP preserving in the discrete sense with a mild restriction on the stabilizing parameter $S_{1}$, stated in the following lemma.

\begin{lemma}[\cite{STY16,TY16}]\label{lem1}
Assume that $\|\vPhi^{0}\|_{\infty}\leq 1$ and the stabilizing parameter $S_{1}$ satisfies
\bq\label{eqn1_3}
S_{1}\geq \max_{\rho\in[-1,1]}\big( M'(\rho)F'(\rho)+M(\rho)F''(\rho)\big).
\eq
For the BDF1 scheme \eqref{BDF_1}, it holds that $\|\vPhi^{n+1}\|_{\infty}\leq1$ for $n=0,1,\cdots,N-1$.
Furthermore, in the case of the mobility function $M(\phi)\equiv1$, we have
\bq
E_{h}(\Phi^{n+1})\leq E_{h}(\Phi^{n}),\quad  \forall\,n=0,1,\cdots,N-1,
\eq
 provided that $S_{1}\geq2$.
\end{lemma}

{\color{black}We} are now ready to present a fully-discrete linear second-order Crank-Nicolson (CN) scheme with two stabilizing terms
for the Allen-Cahn equation with general mobility \eqref{prob}, which reads: given $\Phi^0=\Pi_{\mathcal{C}_h}\phi_0$, and for $n=1,2\cdots,N-1$, find $\Phi^{n+1}\in\mathcal{C}_{h}$ such that
	\begin{subequations}\label{CN_2}
	      \begin{empheq}[left=\empheqlbrace]{align}
		&\Phi^{n+\frac{1}{2}} = {\rm BDF1}(\Phi^n,\tau_{n+1}/2),\label{eqn41}\\
		&\frac{\Phi^{n+1}-\Phi^{n}}{\Dt_{n+1}}-\varepsilon^{2}M(\Phi^{n+\frac{1}{2}})\Delta_{h}\frac{\Phi^{n+1}+\Phi^{n}}{2}+f(\Phi^{n+\frac{1}{2}}) \nonumber\\
  &\qquad\qquad\qquad\qquad+S_{1}\Big(\frac{\Phi^{n+1}+\Phi^{n}}{2}-\Phi^{n+\frac{1}{2}}\Big)+S_{2}\Dt_{n+1}(\Phi^{n+1}-\Phi^{n})=0,\label{eqn4}
		\end{empheq}
       \end{subequations}
       where  the constants $S_{1}$ and $S_2$ are two nonnegative constant  stabilizing parameters.
 The linear doubly stabilized CN scheme \eqref{CN_2} also can be rewritten in the following vector form, as follows:
 	\begin{subequations}\label{CN2_tsor}
	      \begin{empheq}[left=\empheqlbrace]{align}
		&\vPhi^{n+\frac{1}{2}} = {\rm BDF1}(\vPhi^n,\tau_{n+1}/2)\label{CN2_tsor1}\\
		&\frac{\vPhi^{n+1}-\vPhi^{n}}{\Dt_{n+1}}-\varepsilon^{2}\Lambda^{n+\frac{1}{2}}D_{h}\frac{\vPhi^{n+1}+\vPhi^{n}}{2}+f(\vPhi^{n+\frac{1}{2}}) \nonumber\\
  &\qquad\qquad\qquad\qquad+S_{1}\Big(\frac{\vPhi^{n+1}+\vPhi^{n}}{2}-\vPhi^{n+\frac{1}{2}}\Big)+S_{2}\Dt_{n+1}(\vPhi^{n+1}-\vPhi^{n})=0,\label{CN2_tsor2}
		\end{empheq}
       \end{subequations}
      where $\Lambda^{n+\frac{1}{2}}=\mbox{diag}(M(\vPhi^{n+\frac{1}{2}}))$.

  \subsection{Discrete maximum bound principle}
Let us first recall some useful lemmas needed for the analysis of the discrete MBP for the proposed scheme \eqref{CN_2}.
\begin{lemma}[\cite{TY16,LTZ20,HTY17}]\label{lemm2}
Suppose that $B= (b_{i,j})$ is a real $P\times P$ matrix satisfying
\beq
 b_{i,i}<0,\quad |b_{i,i}|\geq\tps\sum_{j\neq i}^{P}|b_{i,j}|,\ \ \ {\color{black}i=1,2,\cdots,P.}
\eeq
Let  $A=aI-B$ where $a>0$ is a constant, then
\beq
\|A\overrightarrow{U}\|_{\infty}\geq a \|\overrightarrow{U}\|_{\infty},\quad \forall\, \overrightarrow{U}\in\R ^{P}.
\eeq
\end{lemma}
\begin{lemma}[\cite{TY16,HJQ22}]\
\label{para}
If the  stabilizing parameter $S_{1}$ satisfies \eqref{eqn1_3}, then it holds
\bq\label{eqn1_4}
\big|S_{1}\rho-f(\rho)\big|\leq S_{1},\quad\forall\,\rho\in[-1,1].
\eq
\end{lemma}
Next, we study the MBP preservation of the proposed CN scheme \eqref{CN_2} in the following theorem.
\begin{theorem}\label{thmMBP}
Assume that  the stabilizing parameter $S_{1}$ satisfies \eqref{eqn1_3} and  $\|\vPhi^{0}\|_{\infty}\leq 1$.
When $S_{2}=0$, the CN scheme \eqref{CN_2} is  conditionally MBP-preserving in the sense that if
\bq\label{res_tau}
 \dps \Dt_{n+1}\leq \frac{2}{S_{1}+4L\varepsilon^{2}/h^{2}}
 \eq
with $L:=\max_{\rho\in[-1,1]}M(\rho)$, then
   $\|\vPhi^{n+1}\|_{\infty}\leq1$ for all $n=0,1,\cdots,N-1$.
   When
\bq\label{res_s}
 \dps S_{2}\geq \Big(\frac{S_1}{4}+\frac{L\varepsilon^{2}}{h^{2}}\Big)^{2},
 \eq
 the CN scheme \eqref{CN_2} is unconditionally MBP-preserving.
\end{theorem}
\begin{proof}
 For any $1\leq n\leq N-1$, we assume $\|\vPhi^{k}\|_{\infty}\leq1$ for  $1\leq k\leq n.$
Using $\vPhi^{n+\frac{1}{2}} = {\rm BDF1}(\vPhi^n,\tau_{n+1}/2)$, $\|\vPhi^{n}\|_{\infty}\leq 1,$ and Lemma \ref{lem1}, we obtain $\|\vPhi^{n+\frac{1}{2}}\|_{\infty}\leq1.$
Thus, together with \eqref{CN2_tsor2}, Lemmas \ref{lemm2} and \ref{para}, we have
\brr\label{MBP}
&\dps\Big(\frac{1}{\Dt_{n+1}}+\frac{S_{1}}{2}+S_{2}\Dt_{n+1}\Big)\|\vPhi^{n+1}\|_{\infty}\\
&~~\leq\dps\Big\|\Big(\big(\frac{1}{\Dt_{n+1}}+\frac{S_{1}}{2}+S_{2}\Dt_{n+1}\big)I-\frac{\varepsilon^{2}}{2}\Lambda^{n+\frac{1}{2}}D_{h}\Big)\vPhi^{n+1}\Big\|_{\infty}\\[8pt]
&~~=\dps\|Q^{n+1}\vPhi^{n}+S_{1}\vPhi^{n+\frac{1}{2}}-f(\vPhi^{n+\frac{1}{2}})\|_{\infty}\\[4pt]
&~~\leq\dps\|Q^{n+1}\|_{\infty}\|\vPhi^{n}\|_{\infty}+\|S_{1}\vPhi^{n+\frac{1}{2}}-f(\vPhi^{n+\frac{1}{2}})\|_{\infty}\\[4pt]
&~~\leq\dps\|Q^{n+1}\|_{\infty}\|\vPhi^{n}\|_{\infty}+S_{1},\\
 \err
 where
 \bq\label{posi}
 Q^{n+1}:=\Big(\frac{1}{\Dt_{n+1}}-\frac{S_{1}}{2}+S_{2}\Dt_{n+1}\Big)I+\frac{\varepsilon^{2}}{2}\Lambda^{n+\frac{1}{2}}D_{h}.
 \eq
If $S_{2}=0$, it follows from \eqref{res_tau}, \eqref{posi}, and the definition of  $\Lambda^{n+\frac12}$ and $D_{h}$ that
\beq
 Q^{n+1}\geq 0,
\eeq
 which means that all the entries of $Q^{n+1}$ are nonnegative.
 If $S_{2}$ satisfies \eqref{res_s}, we can use \eqref{posi} and the definition of  $\Lambda^{n+\frac12}$ and $D_{h}$  to obtain
 \beq
  Q^{n+1}\geq \dps\Big(2\sqrt{S_{2}}-\frac{S_{1}}{2}\Big)I+\frac{\varepsilon^{2}}{2}\Lambda^{n+\frac{1}{2}}D_{h}\geq \frac{2L\varepsilon^{2}}{h^{2}}I+\frac{\varepsilon^{2}}{2}\Lambda^{n+\frac{1}{2}}D_{h}\geq0.
 \eeq
 Thus, it follows  for both of the above choice of $S_2$ that
 \bq\label{est_q}
 \|Q^{n+1}\|_{\infty}\leq\dps \frac{1}{\Dt_{n+1}}-\frac{S_{1}}{2}+S_{2}\Dt_{n+1},
 \eq
 where we have used the fact
$ \sum_{j=1}^{ M^{2}}\big(\Lambda^{n+\frac{1}{2}}D_{h}\big)_{i,j}=0$ for any $1\leq i\leq M^{2}.$
Combining \eqref{MBP} and \eqref{est_q} gives
\bry
\dps\Big(\frac{1}{\Dt_{n+1}}+\frac{S_{1}}{2}+S_{2}\Dt_{n+1}\Big)\|\vPhi^{n+1}\|_{\infty} \leq\;&\dps\|Q^{n+1}\|_{\infty}\|\vPhi^{n}\|_{\infty}+S_{1}\\
\leq\;&\dps\frac{1}{\Dt_{n+1}}+\frac{S_{1}}{2}+S_{2}\Dt_{n+1},
 \ery
 which leads to $\|\vPhi^{n+1}\|_{\infty}\leq1$.
\end{proof}
\begin{remark}
The requirement \eqref{res_s} implies that the selection of the stabilizing parameter $S_{2}$ depends on the size of spatial mesh size $h$.
For practical simulation problems, the interface width parameter $\varepsilon$ is  rather small and $h$ is  usually set  to be the same level of $\varepsilon$ to capture the phase interface, i.e., $\varepsilon/h=O(1)$, thus $S_{2}$ needs not be too large.
\end{remark}

\section{Error analysis and energy stability}
\setcounter{equation}{0}

In this section, we perform error analysis and energy stability of the proposed CN scheme \eqref{CN_2} for the Allen-Cahn equation \eqref{prob} with a
general mobility. Note below that $C$ and $C_i$ denote some generic positive constants independent of $h$ and $\Dt_{n}$.
\subsection{ Discrete $H^{1}$ error estimate and energy stability for the constant mobility case}

In this subsection, the discrete $H^{1}$ error estimate and energy stability of the CN scheme \eqref{CN_2} are investigated for the constant mobility case.
Without loss of generality, we assume $M(\phi)\equiv1$, which leads to $f(\phi) = F'(\phi) = \phi^3-\phi$, the condition \eqref{eqn1_3} being $S_{1}\geq 2$, and the condition \eqref{res_s}  being $ S_{2}\geq \Big(\frac{S_1}{4}+\frac{\varepsilon^{2}}{h^{2}}\Big)^{2}$.

Define the error functions $e^{n}=\Phi^{n}-\Phi(t_{n})$ and $e^{n+\frac{1}{2}}=\Phi^{n+\frac{1}{2}}-\Phi(t_{n+\frac{1}{2}})$ with $\Phi(t)= \Pi_{\mathcal{C}_h}\phi(t)$.
Then, a discrete $H^{1}$ error estimate for the CN scheme \eqref{CN_2} is established in the following theorem with a reasonable regularity requirement on the exact solution $\phi$.
\begin{theorem}\label{th1}
 Assume that $S_{1}\geq 2$, $S_{2}\geq \Big(\frac{S_1}{4}+\frac{\varepsilon^{2}}{h^{2}}\Big)^{2}$, and
$$\phi\in W^{3,\infty}(0,T;L^{\infty}(\Omega))\cap L^{\infty}(0,T;W^{4,\infty}(\Omega)).$$
Then it holds for the CN scheme \eqref{CN_2} in the constant mobility case that
\bq\label{equa4_1}
\dps\varepsilon^{2}\|\nabla_{h} e^{n+1}\|^{2}_{h}+S_{1}\|e^{n+1}\|_h^{2}\leq \dps C_{1}\exp(C_{2}T)\Big(\Dt^{4}+h^{4}\Big)
\eq
for all $0\leq n\leq N-1$.
\end{theorem}

\begin{proof}
We use $\|\Phi\|_{\infty}\leq1, \|\Phi^{n}\|_{\infty}\leq1$ (by the discrete MBP stated in Theorem \ref{thmMBP}), and $f(\cdot)\in C^{1}(\R)$ to obtain that
\bq\label{equ1}
\max\{\|f(\Phi)\|_{\infty},\|f^{'}(\Phi)\|_{\infty},\|f(\Phi^{n})\|_{\infty},\|f^{'}(\Phi^{n})\|_{\infty}\}\leq C_{3}
\eq
 for all $n=0,1,\cdots, N.$
It follows from \eqref{prob} and \eqref{CN_2} that the error equations of $e^{n+\frac{1}{2}}$ and $e^{n+1}$ read
 \begin{subequations}\label{err}
  \begin{align}
 &\frac{ e^{n+\frac{1}{2}}-e^{n}}{\Dt_{n+1}/2}-\varepsilon^{2}\Delta_{h} e^{n+\frac{1}{2}}+S_{1} e^{n+\frac{1}{2}}\nonumber\\
 &\qquad=S_{1} e^{n}-S_{1}(\Phi(t_{n+\frac{1}{2}})-\Phi(t_{n}))
 +f(\Phi(t_{n+\frac{1}{2}}))-f(\Phi^{n})+T_{1}^{n}+T_{2}^{n},\label{err_2}\\
&\dps \frac{e^{n+1}-e^{n}}{\Dt_{n+1}}-\varepsilon^{2}\Delta_{h}\frac{ e^{n+1}+e^{n}}{2}+S_{1}\frac{e^{n+1}+e^{n}}{2}+S_{2}\Dt_{n+1}(e^{n+1}-e^{n})\nonumber\\
&\qquad=S_{1} e^{n+\frac{1}{2}}-S_{1}\Big(\frac{\Phi(t_{n+1})+\Phi(t_{n})}{2}-\Phi(t_{n+\frac{1}{2}})\Big) \label{err_1}\\
&\qquad\quad-S_{2}\Dt_{n+1}(\Phi(t_{n+1})-\Phi(t_{n}))+f(\Phi(t_{n+\frac{1}{2}}))-f(\Phi^{n+\frac{1}{2}})+T_{3}^{n}+T_{4}^{n}\nonumber
 \end{align}
 \end{subequations}
 for $n=1,2,\cdots,N-1$. Let us denote the righthand sides of the above two equalities as $R^{n}_{1}$ and $R^{n}_{2}$, respectively.
For simplicity of expression, we also define the following  error terms $\{T^{n}_{i}\}_{i=1}^4$:
 \brr\label{truerrs}
 T^{n}_{1}&\tps=\Phi_{t}(t_{n+\frac{1}{2}})-\frac{\Phi(t_{n+\frac{1}{2}})-\Phi(t_{n})}{\Dt_{n+1}/2},\quad
  T^{n}_{2}\tps=\varepsilon^2\Delta \Phi(t_{n+\frac{1}{2}})-\varepsilon^2\Delta_{h}\Phi(t_{n+\frac{1}{2}}),\\[5pt]
T^{n}_{3}&\tps=\Phi_{t}(t_{n+\frac{1}{2}})-\frac{\Phi(t_{n+1})-\Phi(t_{n})}{\Dt_{n+1}},\quad
  T^{n}_{4}\tps=\varepsilon^2\Delta \Phi(t_{n+\frac{1}{2}})-\varepsilon^2\Delta_{h}\frac{\Phi(t_{n+1})+\Phi(t_{n})}{2},\\
\err
Taking the discrete $L^{2}$ inner products of \eqref{err_2} and \eqref{err_1} with $\Dt_{n+1}e^{n+\frac{1}{2}}$ and $2(e^{n+1}-e^{n})$, respectively, and using Lemma \ref{intpat} and the identity $2a(a-b)=a^{2}-b^{2}+(a-b)^{2}$, we obtain
 \begin{subequations}\label{err1}
 \begin{align}
 &\dps \|e^{n+\frac{1}{2}}\|^{2}_{h}-\|e^{n}\|^{2}_{h}+\Dt_{n+1}\varepsilon^{2}\|\nabla_{h} e^{n+\frac{1}{2}}\|^{2}_{h}+\Dt_{n+1}S_{1} \|e^{n+\frac{1}{2}}\|^{2}_{h}\leq\dps \Dt_{n+1}\big<R^{n}_{1},e^{n+\frac{1}{2}}\big>_{\Omega},\label{err1_2}\\[5pt]
 &\dps {\color{black}\frac{2\|e^{n+1}-e^{n}\|^{2}_{h}}{\Dt_{n+1}}}+\varepsilon^{2}\big[\|\nabla_{h} e^{n+1}\|^{2}_{h}-\|\nabla_{h} e^{n}\|^{2}_{h}\big]+S_{1}\big[ \|e^{n+1}\|^{2}_{h}-\|e^{n}\|^{2}_{h}\big]\nonumber\\
 &\qquad\dps+S_{2}\Dt_{n+1}\|e^{n+1}-e^{n}\|^{2}_{h}\leq2\big<R^{n+1}_{2},e^{n+1}-e^{n}\big>_{\Omega}.\label{err1_1}
 \end{align}
 \end{subequations}
 From \eqref {equ1}, it follows that
 $$\|f(\Phi(t_{n+\frac{1}{2}}))-f(\Phi^{n+\frac{1}{2}})\|^{2}_{h}\leq (C_{3})^2\|e^{n+\frac{1}{2}}\|^{2}_{h}.$$
Then we obtain the following estimate for the righthand side of \eqref{err1_1} by using Cauchy-Schwarz inequality and Young's inequality
\brr\label{eqnn2}
&2\big<R^{n+1}_{2},e^{n+1}-e^{n}\big>_{\Omega}\\
&\quad\dps\leq \frac{\Dt_{n+1}\|R^{n+1}_{2}\|^{2}_{h}}{2}+\frac{2\|e^{n+1}-e^{n}\|^{2}_{h}}{\Dt_{n+1}}\\
&\quad\leq\dps \Big(S^{2}_{1}\|e^{n+\frac{1}{2}}\|^{2}_{h}+\frac{S^{2}_{1}|\Omega|\|\phi\|^{2}_{W^{2,\infty}(0,T;L^{\infty}(\Omega))}}{16}\Dt_{n+1}^{4}+S^{2}_{2}|\Omega|\|\phi\|^{2}_{W^{1,\infty}(0,T;L^{\infty}(\Omega))}\Dt^{4}_{n+1}\\
&\quad\quad\dps+(C_{3})^2\|e^{n+\frac{1}{2}}\|^{2}_{h}+\|T^{n}_{3}\|^{2}_{h}+\|T^{n}_{4}\|^{2}_{h}\Big)\frac{\Dt_{n+1}}{2}+\frac{2\|e^{n+1}-e^{n}\|^{2}_{h}}{\Dt_{n+1}}\\[5pt]
&\quad\leq\dps C_{4}\Dt_{n+1}\|e^{n+\frac{1}{2}}\|^{2}_{h}+\frac{\Dt_{n+1}}{2}(\|T^{n}_{3}\|^{2}_{h}+\|T^{n}_{4}\|^{2}_{h})+C_{5}\Dt^{5}_{n+1}+\frac{2\|e^{n+1}-e^{n}\|^{2}_{h}}{\Dt_{n+1}},
\err
where $C_{4}:=(S^{2}_{1}+(C_{3})^{2})/2$ and
\beq
C_{5}=\frac{S^{2}_{1}|\Omega|\|\phi\|^{2}_{W^{2,\infty}(0,T;L^{\infty}(\Omega))}}{32}+\frac{S^{2}_{2}|\Omega|}{2}\|\phi\|^{2}_{W^{1,\infty}(0,T;L^{\infty}(\Omega))}.
\eeq
Therefore, we deduce from \eqref{err1_1} and \eqref{eqnn2} that
\brr\label{eqt1}
&\dps\varepsilon^{2}\big[\|\nabla_{h} e^{n+1}\|^{2}_{h}-\|\nabla_{h} e^{n}\|^{2}_{h}\big]+S_{1}\big[ \|e^{n+1}\|^{2}_{h}-\|e^{n}\|^{2}_{h}\big]\\[11pt]
&\qquad\leq \dps C_{4}\Dt_{n+1}\|e^{n+\frac{1}{2}}\|^{2}_{h}+\frac{\Dt_{n+1}}{2}(\|T^{n}_{3}\|^{2}_{h}+\|T^{n}_{4}\|^{2}_{h})+C_{5}\Dt^{5}_{n+1}.
\err
Similarly, we can obtain the following estimate from \eqref{err1_2}
\brr\label{eqnn1}
&\dps\|e^{n+\frac{1}{2}}\|^{2}_{h}-\|e^{n}\|^{2}_{h}+\Dt_{n+1}\varepsilon^{2}\|\nabla_{h} e^{n+\frac{1}{2}}\|^{2}_{h}+\Dt_{n+1}S_{1} \|e^{n+\frac{1}{2}}\|^{2}_{h}\\[5pt]
&\qquad\leq\dps \frac{\Dt_{n+1}^{2}}{2}\Big([S^{2}_{1}+(C_{3})^{2}]\|e^{n}\|^{2}_{h}+\frac{[S^{2}_{1}+(C_{3})^{2}]|\Omega|\|\phi\|^{2}_{W^{1,\infty}(0,T;L^{\infty}(\Omega))}}{4}\Dt_{n+1}^{2}\\[5pt]
&\qquad\quad\dps+\|T_{1}^{n}\|^{2}_{h}+\|T_{2}^{n}\|^{2}_{h}\Big)+\tps\frac12{\|e^{n+\frac{1}{2}}\|^{2}_{h}}\\
\err
where we have used
\bry
\|f(\Phi(t_{n+\frac{1}{2}}))-f(\Phi^{n})\|^{2}_{h}\leq\;&\dps\|f(\Phi(t_{n+\frac{1}{2}}))-f(\Phi(t_{n}))\|^{2}_{h}+\|f(\Phi(t_{n}))-f(\Phi^{n})\|^{2}_{h}\\[4pt]
\leq\;&\dps (C_{3})^2\big(\|\Phi(t_{n+\frac{1}{2}})-\Phi(t_{n})\|^{2}_{h}+\|e^{n}\|^{2}_{h}\big)\\[5pt]
\leq\;&\dps (C_{3})^2\Big(\frac{\Dt_{n+1}^{2}}{4}|\Omega|\|\phi\|^{2}_{W^{1,\infty}(0,T;L^{\infty}(\Omega))}+\|e^{n}\|^{2}_{h}\Big).
\ery
Then it follows from \eqref{eqnn1} and the definition of $C_{4}$ in \eqref{eqnn2} that
\brr\label{eqt2}
&\dps \|e^{n+\frac{1}{2}}\|^{2}_{h}
\leq\dps 2\|e^{n}\|^{2}_{h}+\Dt^{2}_{n+1}\big(2C_{4}\|e^{n}\|^{2}_{h}+C_{6}\Dt^{2}_{n+1}+\|T_{1}^{n}\|^{2}_{h}+\|T_{2}^{n}\|^{2}_{h}\big)
\err
with $C_{6}:=\frac{C_{4}}{2}|\Omega|\|\phi_{t}\|^{2}_{W^{1,\infty}(0,T;L^{\infty}(\Omega))}$.
Substituting the estimate \eqref{eqt2} for $e^{n+\frac{1}{2}}$ into \eqref{eqt1}, we have
\brr\label{eqt3}
&\varepsilon^{2}\big[\|\nabla_{h} e^{n+1}\|^{2}_{h}-\|\nabla_{h} e^{n}\|^{2}_{h}\big]+S_{1}\big[ \|e^{n+1}\|^{2}_{h}-\|e^{n}\|^{2}_{h}\big]\\[9pt]
&\qquad\leq\dps C_7\Dt_{n+1}\|e^{n}\|^{2}_{h}+(C_{5}+C_{4}C_{6})\Dt^{5}_{n+1}+C_{4}\Dt^{3}_{n+1}(\|T_{1}^{n}\|^{2}_{h}+\|T^{n}_{2}\|^{2}_{h})\\[5pt]
&\qquad\quad\dps+\frac{\Dt_{n+1}}{2}(\|T^{n}_{3}\|^{2}_{h}+\|T^{n}_{4}\|^{2}_{h}),
\err
where $C_{7}:=2C_{4}(1+C_{4}\Dt^{2}).$
For the truncation errors $T^{n}_{i}, i=1,2,3,4$, we have the following estimates (see \cite{LSR19,LTZ20}):
\brr\label{equ8}
\|T^{n}_{1}\|^{2}_{h}\leq &\;\dps \frac{|\Omega|}{16}\Dt_{n+1}^{2}\|\phi\|^{2}_{W^{2,\infty}(0,T;L^{\infty}(\Omega))},\\[5pt]
\|T^{n}_{2}\|^{2}_{h}\leq &\;\dps \frac{\varepsilon^{4}|\Omega|}{36}h^{4}\|\phi\|^{2}_{L^{\infty}(0,T;W^{4,\infty}(\Omega))},\\[5pt]
\|T^{n}_{3}\|^{2}_{h}\leq &\;\dps \frac{|\Omega|}{24^2}\Dt_{n+1}^{4}\|\phi\|^{2}_{W^{3,\infty}(0,T;L^{\infty}(\Omega))},\\[5pt]
\|T^{n}_{4}\|^{2}_{h}\leq &\;\dps \varepsilon^{4}\Big[\frac{|\Omega|}{64}\Dt_{n+1}^{4}\|\phi\|^{2}_{W^{3,\infty}(0,T;W^{2,\infty}(\Omega))}+\frac{|\Omega|}{36}h^{4}\|\phi\|^{2}_{L^{\infty}(0,T;W^{4,\infty}(\Omega))}\Big].
\err
Thus, we sum up the inequality \eqref{eqt3} from 0 to $n$ to derive that
\bq\label{equ9}
\dps\varepsilon^{2}\|\nabla_{h}e^{n+1}\|^{2}_{h}+S_{1}\|e^{n+1}\|^{2}_{h}\leq\dps C_{7}\sum_{k=1}^{n}\Dt_{k+1}\|e^{k}\|^{2}_{h}+C_{8}\Dt^{4}+C_{9}h^{4},
\eq
where
\bry
C_{8}:=\;&\dps T\Big[C_{5}+C_{4}C_{6}+ \frac{C_{4}|\Omega|}{16}\|\phi\|^{2}_{W^{2,\infty}(0,T;L^{\infty}(\Omega))}+\frac{|\Omega|}{2\times24^2}\|\phi\|^{2}_{W^{3,\infty}(0,T;L^{\infty}(\Omega))}\\
&\dps+\frac{|\Omega|}{128}\varepsilon^{4}\|\phi\|^{2}_{W^{3,\infty}(0,T;W^{2,\infty}(\Omega))}\Big],\\
C_{9}:=\;&\dps \varepsilon^{4}T\Big[C_{4}T+\frac{1}{2}\Big]\frac{|\Omega|}{36}\|\phi\|^{2}_{L^{\infty}(0,T;W^{4,\infty}(\Omega))}.
\ery
Using \eqref{equ9} and the discrete Gronwall's lemma, we then obtain the desired estimate \eqref{equa4_1}.
 \end{proof}

The energy stability of the proposed CN scheme \eqref{CN_2} is established in the following theorem by using its the MBP property (Theorem \ref{thmMBP}) and the discrete $H^1$ error estimate (Theorem \ref{th1}).
\begin{theorem}\label{st_the}
Under the assumption of Theorem \ref{th1}, the CN scheme \eqref{CN_2} in the constant mobility case satisfies
\bq\label{eq2_1}
\dps E_{h}(\Phi^{n+1})-E_{h}(\Phi^{n})\leq C\Dt_{n+1}(h^{4}+\Dt^{2}),
\eq
and consequently,
\bq\label{equn2}
\dps E_{h}(\Phi^{n+1})\leq E_{h}(\Phi^{0})+C T(h^{4}+\Dt^{2})
\eq
for all $0\leq n\leq N-1$,  i.e., the discrete free energy is uniformly  bounded by the energy at the initial time plus a constant.
\end{theorem}
\begin{proof}
We take the discrete $L^{2}$-inner product of \eqref{eqn4} with $\Phi^{n+1}-\Phi^{n}$, to obtain that
\brr\label{eqt4}
&\dps \big[\frac{1}{\Dt_{n+1}}+S_{2}\Dt_{n+1}\big]\|\Phi^{n+1}-\Phi^{n}\|^{2}_{h}+\frac{\varepsilon^{2}}{2}\big(\|\nabla_{h}\Phi^{n+1}\|^{2}_{h}-\|\nabla_{h}\Phi^{n}\|^{2}_{h}\big)\\[5pt]
&\qquad=\dps-\big<f(\Phi^{n+\frac{1}{2}}),\Phi^{n+1}-\Phi^{n}\big>_{\Omega}-S_{1}\big<\frac{\Phi^{n+1}+\Phi^{n}}{2}-\Phi^{n+\frac{1}{2}},\Phi^{n+1}-\Phi^{n}\big>_{\Omega}.
\err
Noting that
\brr\label{equnt1}
 (a^{3}-a)(a-b)&=\dps a^{3}(a-b)-a(a-b)\\
 &\geq\dps \frac{a^{4}-b^{4}}{4}-\frac{1}{2}\big(a^{2}-b^{2}+(a-b)^{2}\big)\\[5pt]
 &=\dps\frac{(a^{2}-1)^{2}}{4}-\frac{(b^{2}-1)^{2}}{4}-\frac{(a-b)^{2}}{2},
 \err
 we deduce
 \bq\label{eqnt2}
 \big<F(\Phi^{n+1})-F(\Phi^{n}),1\big>_{\Omega}\leq\dps \big<f(\Phi^{n+1}), \Phi^{n+1}-\Phi^{n}\big>_{\Omega}+\frac{1}{2}\|\Phi^{n+1}-\Phi^{n}\|^{2}_{h}.
 \eq
 From \eqref{eqt4}, \eqref{eqnt2} and the definition of $E_{h}(\Phi^{n})$, it {\color{black}follows} that
 \bry
 &E_{h}(\Phi^{n+1})-E_{h}(\Phi^{n})\\
 &\quad=\dps\frac{\varepsilon^{2}}{2}\big(\|\nabla_{h}\Phi^{n+1}\|^{2}_{h}-\|\nabla_{h}\Phi^{n}\|^{2}_{h}\big)+\big<F(\Phi^{n+1})-F(\Phi^{n}),1\big>_{\Omega}\\
 &\quad\leq\dps\big<f(\Phi^{n+1})-f(\Phi^{n+\frac{1}{2}}),\Phi^{n+1}-\Phi^{n}\big>_{\Omega}-S_{1}\big<\frac{\Phi^{n+1}+\Phi^{n}}{2}-\Phi^{n+\frac{1}{2}},\Phi^{n+1}-\Phi^{n}\big>_{\Omega}\\[5pt]
 &\quad\quad\dps-\big[\frac{1}{\Dt_{n+1}}+S_{2}\Dt_{n+1}-\frac{1}{2}\big]\|\Phi^{n+1}-\Phi^{n}\|^{2}_{h}.
 \ery
 Furthermore, using the Cauchy-Schwarz inequality and Young's inequality, we obtain
  \brr\label{eqnt3}
 &E_{h}(\Phi^{n+1})-E_{h}(\Phi^{n})\\
 &\quad\leq\dps (C_{3})^{2}\Dt_{n+1}\|\Phi^{n+1}-\Phi^{n+\frac{1}{2}}\|^{2}_{h}+S^{2}_{1}\Dt_{n+1}\Big\|\frac{\Phi^{n+1}+\Phi^{n}}{2}-\Phi^{n+\frac{1}{2}}\Big\|^{2}_{h}\\[5pt]
 &\quad\quad\dps+\frac{\|\Phi^{n+1}-\Phi^{n}\|^{2}_{h}}{2\Dt_{n+1}}-\big[\frac{1}{\Dt_{n+1}}+S_{2}\Dt_{n+1}-\frac{1}{2}\big]\|\Phi^{n+1}-\Phi^{n}\|^{2}_{h}\\
 &\quad=\dps(C_{3})^{2}\Dt_{n+1}\|\Phi^{n+1}-\Phi^{n+\frac{1}{2}}\|^{2}_{h}+S^{2}_{1}\Dt_{n+1}\Big\|\frac{\Phi^{n+1}+\Phi^{n}}{2}-\Phi^{n+\frac{1}{2}}\Big\|^{2}_{h}\\[5pt]
 &\quad\quad\dps-\big[\frac{1}{2\Dt_{n+1}}+S_{2}\Dt_{n+1}-\frac{1}{2}\big]\|\Phi^{n+1}-\Phi^{n}\|^{2}_{h}.
 \err
 By the triangle inequality, we obtain that
 \brr\label{eqnt4}
 \|\Phi^{n+1}-\Phi^{n+\frac{1}{2}}\|^{2}_{h}=\;&\tps\|e^{n+1}+\Phi(t_{n+1})-\Phi(t_{n+\frac{1}{2}})-e^{n+\frac{1}{2}}\|^{2}_{h}\\[5pt]
 \leq \;&\tps\|e^{n+1}\|^{2}_{h}+\|e^{n+\frac{1}{2}}\|^{2}_{h}+\frac{\Dt^{2}_{n+1}|\Omega|}{4}\|\phi_{t}\|^{2}_{L^{\infty}(0,T;L^{\infty}(\Omega))},\\[5pt]
\tps \Big\|\frac{\Phi^{n+1}+\Phi^{n}}{2}-\Phi^{n+\frac{1}{2}}\Big\|^{2}_{h}=\;&\tps \Big\|\frac{e^{n+1}+e^{n}}{2}+\frac{\Phi(t_{n+1})+\Phi(t_{n})}{2}-\Phi(t_{n+\frac{1}{2}})+e^{n+\frac{1}{2}}\Big\|^{2}_{h}\\[5pt]
 \leq\;&\tps\|e^{n+1}\|^{2}_{h}+\|e^{n}\|^{2}_{h}+\|e^{n+\frac{1}{2}}\|^{2}_{h}+\frac{\Dt^{2}_{n+1}|\Omega|}{4}\|\phi_{t}\|^{2}_{L^{\infty}(0,T;L^{\infty}(\Omega))}.
 \err
Thus, it follows from \eqref{eqnt3} and \eqref{eqnt4} that
\brr\label{equn1}
E_{h}(\Phi^{n+1})-E_{h}(\Phi^{n})\leq\;&\dps C_{10}\Dt_{n+1}\big(\|e^{n}\|^{2}_{h}+\|e^{n+1}\|^{2}_{h}+\|e^{n+\frac{1}{2}}\|^{2}_{h}+\Dt^{2}_{n+1}\big)\\[4pt]
&\dps-\big[\frac{1}{2\Dt_{n+1}}+S_{2}\Dt_{n+1}-\frac{1}{2}\big]\|\Phi^{n+1}-\Phi^{n}\|^{2}_{h}.
 \err
 with
 \beq
 C_{10}=\dps\max\Big\{(C_{3})^{2},S^{2}_{1},\frac{(C_{3})^{2}|\Omega|}{4}\|\phi_{t}\|^{2}_{L^{\infty}(0,T;L^{\infty}(\Omega))},\frac{S^{2}_{1}|\Omega|}{4}\|\phi_{t}\|^{2}_{L^{\infty}(0,T;L^{\infty}(\Omega))}\Big\}.
 \eeq
Furthermore, since $S_{1}\geq 2$  and $S_{2}$ satisfying \eqref{res_s} in the constant mobility case, we  get
\bq\label{eqnt5}
\frac{1}{2\Dt_{n+1}}+S_{2}\Dt_{n+1}-\frac{1}{2}\geq 2\sqrt{\frac{S_{2}}{2}}-\frac{1}{2}>0.
\eq
Together with \eqref{equa4_1} and \eqref{eqt2}, we then obtain
\brr\label{equn1_1}
E_{h}(\Phi^{n+1})-E_{h}(\Phi^{n})&\leq\dps C_{10}\Dt_{n+1}\big(\|e^{n}\|^{2}_{h}+\|e^{n+1}\|^{2}_{h}+\|e^{n+\frac{1}{2}}\|^{2}_{h}+\Dt^{2}_{n+1}\big)\\[4pt]
&\leq\dps C\Dt_{n+1}(h^{4}+\Dt^{2}).
 \err
 Summing up the above inequality from 0 to $n$ gives the desired result \eqref{equn2}.
\end{proof}


\subsection{$L^{\infty}$ error estimate and energy stability for the general mobility case}

In this subsection, the discrete  $L^{\infty}$ error estimate and  energy stability of the CN scheme \eqref{CN_2} are investigated for the  case with a  general mobility $M(\phi)\geq M_0>0$.

\begin{theorem}\label{th3}
Assume $M(\cdot)\in C^{1}(\R)$, $S_{1}$ satisfies \eqref{eqn1_3}, $S_{2}$ satisfies \eqref{res_s}, and
$$\phi\in W^{3,\infty}(0,T;L^{\infty}(\Omega))\cap L^{\infty}(0,T;W^{4,\infty}(\Omega)).$$

 Then it holds  for the CN scheme \eqref{CN_2} in the general mobility case that
\bq\label{eqnn8}
\|\ve^{n+1}\|_{\infty}
\leq C_1 \exp\big(C_{2}T\big)\big(\Dt^{2}+h^{2}\big)
\eq
for all $0\leq n\leq N-1$.
\end{theorem}
\begin{proof}
The exact solution $\vPhi(\cdot)$ satisfies
\bry
&\dps\frac{\vPhi(t_{n+1})-\vPhi(t_{n})}{\Dt_{n+1}}+\Lambda(\vPhi(t_{n+\frac{1}{2}}))\Big(-\varepsilon^{2}D_{h}\frac{\vPhi(t_{n+1})+\vPhi(t_{n})}{2}+F'(\vPhi(t_{n+\frac{1}{2}}))\Big)\\
&\quad\dps+\overrightarrow{T}^{n}_{3}+\Lambda(\vPhi(t_{n+\frac{1}{2}}))\overrightarrow{T}^{n}_{4}=0
\ery
for any $1\leq n\leq N-1$, where $\Lambda(\vPhi(t_{n+1/2})):=\mbox{diag}(M(\vPhi(t_{n+1/2}))$ and $\overrightarrow{T}^{n}_{3}$ and $\overrightarrow{T}^{n}_{4}$ are the vector forms of $T^{n}_{3}$ and $T^{n}_{4}$ in \eqref{truerrs}, respectively. Moreover, $\overrightarrow{T}^{n}_{4}$ can be expressed as
\bq\label{eqnn3}
 \overrightarrow{T}^{n}_{4}\dps=-\varepsilon^2\Delta \vPhi(t_{n+\frac{1}{2}})+\varepsilon^2 D_{h}\frac{\vPhi(t_{n+1})+\vPhi(t_{n})}{2}.
\eq
Furthermore, it is easy to verify that
\brr\label{eqnn5}
\|\overrightarrow{T}^{n}_{3}\|_{\infty}&\dps\leq  \frac{1}{24}\Dt_{n+1}^{2}\|\phi\|_{W^{3,\infty}(0,T;L^{\infty}(\Omega))},\\[5pt]
\|\overrightarrow{T}^{n}_{4}\|_{\infty}&\dps\leq  \varepsilon^{2}\Big[\frac{\Dt_{n+1}^{2}}{8}\|\phi\|_{W^{3,\infty}(0,T;W^{2,\infty}(\Omega))}+\frac{h^{2}}{6}\|\phi\|^{2}_{L^{\infty}(0,T;W^{4,\infty}(\Omega))}\Big].
\err
Together with \eqref{CN2_tsor2}, the error equation of $\ve^{n+1}$ reads as
\bry
&\dps\frac{\ve^{n+1}-\ve^{n}}{\Dt_{n+1}}+S_{1}\Big(\frac{\ve^{n+1}+\ve^{n}}{2}-\ve^{n+\frac{1}{2}}\Big)+S_{2}\Dt_{n+1}(\ve^{n+1}-\ve^{n})-\varepsilon^{2}\Lambda^{n+\frac{1}{2}}D_{h}\frac{\ve^{n+1}+\ve^{n}}{2}\\[2pt]
&~=\dps S_{1}\tps(\vPhi(t_{n+\frac{1}{2}})-\frac{\vPhi(t_{n+1})+\vPhi(t_{n})}{2})
-S_{2}\Dt_{n+1}\big(\vPhi(t_{n+1})-\vPhi(t_{n})\big)\\
&~~-\Lambda^{n+\frac{1}{2}}\big[F'(\vPhi^{n+\frac{1}{2}})-F'(\vPhi(t_{n+\frac{1}{2}}))\big]\\[5pt]
&~~\tps-\big[\Lambda^{n+\frac{1}{2}}-\Lambda(\vPhi(t_{n+\frac{1}{2}}))\big]\big[-\varepsilon^{2}D_{h}{\color{black}\frac{\vPhi(t_{n+1})+\vPhi(t_{n})}{2}}+F'(\vPhi(t_{n+\frac{1}{2}}))\big]+\overrightarrow{T}^{n}_{3}+\Lambda(\vPhi(t_{n+\frac{1}{2}}))\overrightarrow{T}^{n}_{4}.
\ery
Let us denote the right-hand side term of the above equality as $R^{n}$.
Then, the above equality can be rewritten as:
\bq\label{eqn1_7}
\dps\frac{\ve^{n+1}}{\Dt_{n+1}}+\frac{S_{1}}{2}\ve^{n+1}+S_{2}\Dt_{n+1}\ve^{n+1}-\frac{\varepsilon^{2}}{2}\Lambda^{n+\frac{1}{2}}D_{h}\ve^{n+1}=\dps Q^{n+1}\ve^{n}+ S_{1}\ve^{n+\frac{1}{2}}+R^{n},
\eq
where $Q^{n+1}$ is defined in \eqref{posi}.
Using the estimate for the matrix $Q^{n+1}$ in \eqref{est_q}, we deduce that
\bq\label{eqnn6}
\|Q^{n+1}\ve^{n}\|_{\infty}\leq\Big(\frac{1}{\Dt_{n+1}}-\frac{S_{1}}{2}+S_{2}\Dt_{n+1}\Big)\|\ve^{n}\|_{\infty}\leq\Big(\frac{1}{\Dt_{n+1}}+\frac{S_{1}}{2}+S_{2}\Dt_{n+1}\Big)\|\ve^{n}\|_{\infty}
\eq
From the definition of $F(\rho)=\frac14(1-\rho^{2})^{2}$, it follows that
$\max_{\rho\in[-1,1]}F'(\rho)=\frac{2}{3\sqrt{3}}$ and  $\max_{\rho\in[-1,1]}F''(\rho)=2.$
Moreover, we derive that
\bq
\|F'(\vPhi^{n+\frac{1}{2}})-F'(\vPhi(t_{n+\frac{1}{2}}))\|_{\infty}\leq 2\|\ve^{n+\frac{1}{2}}\|_{\infty}.
\eq
Furthermore, we can use \eqref{eqnn3} and \eqref{eqnn5} to obtain
\brr
&\tps\|-\varepsilon^{2}D_{h}{\color{black}\frac{\vPhi(t_{n+1})+\vPhi(t_{n})}{2}}+F'(\vPhi(t_{n+\frac{1}{2}}))\|_{\infty}\\
&\quad=\tps\|-\varepsilon^{2}\Delta\vPhi(t_{n+\frac{1}{2}})-\overrightarrow{T}^{n}_{4}+F'(\vPhi(t_{n+\frac{1}{2}}))\|_{\infty}\\
&\quad\leq\tps\varepsilon^{2}\|\phi\|_{L^{\infty}(0,T,W^{2,\infty}(\Omega))}+\|\overrightarrow{T}^{n}_{4}\|_{\infty}+\frac{2}{3\sqrt{3}}\\
&\quad\tps\leq\varepsilon^{2}\|\phi\|_{L^{\infty}(0,T,W^{2,\infty}(\Omega))}+\frac{2}{3\sqrt{3}}\\
&\qquad\tps+\varepsilon^{2}\big(\frac{\Dt^{2}}{8}\|\phi\|_{W^{3,\infty}(0,T;W^{2,\infty}(\Omega))}+\frac{h^{2}}{6}\|\phi\|^{2}_{L^{\infty}(0,T;W^{4,\infty}(\Omega))}\big)\\
&\quad=:C_{3}.
\err
The definitions of $\Lambda^{n+1/2}$ and $\Lambda(\cdot)$ give us
\bq\label{eqnn4}
\|\Lambda^{n+\frac{1}{2}}-\Lambda(\vPhi(t_{n+\frac{1}{2}}))\|_{\infty}\leq\dps\max_{\rho\in[-1,1]}\big|M'(\rho)\big|\| \ve^{n+\frac{1}{2}}\|_{\infty}.
\eq
Multiplying \eqref{eqn1_7} with $\Dt_{n+1}$, and combining it with \eqref{eqnn5} and \eqref{eqnn6}-\eqref{eqnn4}, we deduce from Lemma \ref{lemm2} that
\brr\label{eqn1_12}
\|\ve^{n+1}\|_{\infty}\leq&\tps \Big\|\ve^{n+1}+\frac{S_{1}}{2}\Dt_{n+1}\ve^{n+1}+S_{2}\Dt^{2}_{n+1}\ve^{n+1}-\frac{\varepsilon^{2}\Dt_{n+1}}{2}\Lambda^{n+\frac{1}{2}}D_{h}\ve^{n+1}\Big\|_{\infty}\\[5pt]
=&\;\tps \Dt_{n+1}\| Q^{n+1}\ve^{n}+ S_{1}\ve^{n+\frac{1}{2}}+R^{n}\|_{\infty}\\
\leq&\;\tps\|\ve^{n}\|_{\infty}+\Dt_{n+1}\Big[(\frac{S_{1}}{2}+S_{2}\Dt_{n+1})\|\ve^{n}\|_{\infty}+S_{1}\|\ve^{n+\frac{1}{2}}\|_{\infty}\\
&\tps+\frac{S_{1}\|\phi\|_{W^{2,\infty}(0,T;L^{\infty}(\Omega))}}{8}\Dt_{n+1}^{2}+S_{2}\|\phi\|_{W^{1,\infty}(0,T;L^{\infty}(\Omega))}\Dt^{2}_{n+1}+2L\| \ve^{n+\frac{1}{2}}\|_{\infty}\\
&\tps+C_{3}\max_{\rho\in[-1,1]}\big|M'(\rho)\big|\| \ve^{n+\frac{1}{2}}\|_{\infty}+\|\overrightarrow{T}^{n}_{3}\|_{\infty}+L\|\overrightarrow{T}^{n}_{4}\|_{\infty}\Big]\\
\leq&\;\tps\|\ve^{n}\|_{\infty}+\Dt_{n+1}\big[ C_{4}\|\ve^{n}\|_{\infty}+C_{5}\| \ve^{n+\frac{1}{2}}\|_{\infty}\\[5pt]
&\tps+C_6\Dt^{2}_{n+1}+\frac{\varepsilon^{2}Lh^{2}}{6}\|\phi\|_{L^{\infty}(0,T;W^{4,\infty}(\Omega))}\big],
\err
where $C_{4}=S_{1}/2+S_{2}\Dt$, ${C}_{5}=\dps S_{1}+2L+C_{3}\max_{\rho\in[-1,1]}\big|M'(\rho)\big|$, and
\brr\label{eqn1_9}
{C}_{6}=\;&\dps\frac{S_{1}\|\phi\|_{W^{2,\infty}(0,T;L^{\infty}(\Omega))}}{8}+S_{2}\|\phi\|_{W^{1,\infty}(0,T;L^{\infty}(\Omega))}+ \frac{\|\phi\|_{W^{3,\infty}(0,T;L^{\infty}(\Omega))}}{24}\\
&\dps\;+\frac{\varepsilon^{2}L\|\phi\|_{W^{3,\infty}(0,T;W^{2,\infty}(\Omega))}}{8}.
\err
Following the similar process of deriving \eqref{eqn1_7}, we can easily obtain the error equation of $ \ve^{n+\frac{1}{2}}$ from \eqref{prob} and \eqref{CN2_tsor1}:
\brr\label{eqn1_8}
&\dps\frac{{\color{black}2} \ve^{n+\frac{1}{2}}}{\Dt_{n+1}}+S_{1} \ve^{n+\frac{1}{2}}-\varepsilon^{2}\Lambda^{n}D_{h} \ve^{n+\frac{1}{2}}\\
&\quad=\dps\frac{{\color{black}2}\ve^{n}}{\Dt_{n+1}}+S_{1} \ve^{n}-S_{1}(\vPhi(t_{n+\frac{1}{2}})-\vPhi(t_{n}))-\Lambda^{n}\big[F'(\vPhi^{n})-F'(\vPhi(t_{n+\frac{1}{2}}))\big]\\
&\qquad-\big[\Lambda^{n}-\Lambda(\vPhi(t_{n+\frac{1}{2}}))\big]\big[-\varepsilon^{2}D_{h}\vPhi(t_{n+\frac{1}{2}})+F'(\vPhi(t_{n+\frac{1}{2}}))\big]\\[5pt]
&\qquad\dps+\overrightarrow{T}^{n}_{1}+\Lambda(\vPhi(t_{n+\frac{1}{2}}))\overrightarrow{T}^{n}_{2},
\err
where
$\overrightarrow{T}^{n}_{1}$ and $\overrightarrow{T}^{n}_{2}$ are vector forms of $T^{n}_{1}$ and $T^{n}_{2}$, respectively. Moreover, we have
\begin{equation}\label{eqntt1}
\|\overrightarrow{T}^{n}_{1}\|_{\infty}\leq\dps \frac{\Dt_{n+1}}{4}\|\phi\|_{W^{2,\infty}(0,T;L^{\infty}(\Omega))},\quad
\|\overrightarrow{T}^{n}_{2}\|^{2}_{\infty}\leq\dps \frac{\varepsilon^{2}}{6}h^{2}\|\phi\|_{L^{\infty}(0,T;W^{4,\infty}(\Omega))}.
\end{equation}
Using the triangle inequality we get
\brr\label{eqnn9}
&\|F'(\vPhi^{n})-F'(\vPhi(t_{n+\frac{1}{2}}))\|_{\infty}\\
&\quad=\dps\|F'(\vPhi^{n})-F'(\vPhi(t_{n}))+F'(\vPhi(t_{n}))-F'(\vPhi(t_{n+\frac{1}{2}}))\|_{\infty}\\[5pt]
&\quad\leq\dps 2\| \ve^{n}\|_{\infty}+\|\phi\|_{W^{1,\infty}(0,T;L^{\infty}(\Omega))}\Dt_{n+1},\\[5pt]
&\big\|\Lambda^{n}-\Lambda(\vPhi(t_{n+\frac{1}{2}}))\big\|_{\infty}\\
&\quad\leq\dps\big\|\Lambda^{n}-\Lambda(\vPhi(t_{n}))\big\|_{\infty}+\big\|\Lambda(\vPhi(t_{n}))-\Lambda(\vPhi(t_{n+\frac{1}{2}}))\big\|_{\infty}\\[5pt]
&\quad\leq\dps\max_{\rho\in[-1,1]}\big|M'(\rho)\big|\big[\| \ve^{n}\|_{\infty}+\frac{\Dt_{n+1}}{2}\|\phi\|_{W^{1,\infty}(0,T;L^{\infty}(\Omega))}\big].
\err
Multiplying \eqref{eqn1_8} with $\Dt_{n+1}$, and using Lemma \ref{lemm2}, \eqref{eqntt1}, and \eqref{eqnn9}, we obtain that
\bry
{\color{black}2}\| \ve^{n+\frac{1}{2}}\|_{\infty}\leq\;&\dps\big\| {\color{black}2}\ve^{n+\frac{1}{2}}+S_{1}\Dt_{n+1} \ve^{n+\frac{1}{2}}-\varepsilon^{2}\Dt_{n+1}\Lambda^{n}D_{h} \ve^{n+\frac{1}{2}}\big\|_{\infty}\\[5pt]
\leq\;&\tps {\color{black}2}\|\ve^{n}\|_{\infty}+C_{5}\Dt_{n+1}\| \ve^{n}\|_{\infty}+\frac{S_{1}+2L+C_{3}\max_{\rho\in[-1,1]}\big|M'(\rho)\big|}{2}\|\phi\|_{W^{1,\infty}(0,T;L^{\infty}(\Omega))}\Dt^{2}_{n+1}\\[5pt]
\;&\dps+\frac{\Dt^{2}_{n+1}}{4}\|\phi\|_{W^{2,\infty}(0,T;L^{\infty}(\Omega))}+\Dt_{n+1} \frac{\varepsilon^{2}Lh^{2}}{6}\|\phi\|_{L^{\infty}(0,T;W^{4,\infty}(\Omega))}.
\ery
Therefore, we obtain
\bq\label{eqn1_13}
\| \ve^{n+\frac{1}{2}}\|_{\infty}\leq {C}_{7}\| \ve^{n}\|_{\infty}+C_{8}\Dt^{2}_{n+1}+\Dt_{n+1} \frac{\varepsilon^{2}Lh^{2}}{12}\|\phi\|_{L^{\infty}(0,T;W^{4,\infty}(\Omega))},
\eq
where $
{C}_{7}=1+C_{5}\Dt/2$, and
\beq
C_{8}=\frac{S_{1}+2L+C_{3}\max_{\rho\in[-1,1]}\big|M'(\rho)\big|}{4}\|\phi\|_{W^{1,\infty}(0,T;L^{\infty}(\Omega))}+\frac{1}{8}\|\phi\|_{W^{2,\infty}(0,T;L^{\infty}(\Omega))}.
\eeq
Substituting the estimate \eqref{eqn1_13} for $\|e^{n+1/2}\|$ into \eqref{eqn1_12}, gives
\brr\label{eqnn7}
\|\ve^{n+1}\|_{\infty}
\leq\;&\dps\|\ve^{n}\|_{\infty}+\Dt_{n+1}\big[ (C_{4}+C_{5}C_{7})\|\ve^{n}\|_{\infty}+(C_{5}C_{8}+C_{6})\Dt^{2}_{n+1}\\[3pt]
&\;\dps+(C_{5}\frac{\Dt}{2}+1)\frac{\varepsilon^{2}Lh^{2}}{6}\|\phi\|_{L^{\infty}(0,T;W^{4,\infty}(\Omega))}\big]\\[5pt]
\leq\;&\dps\|\ve^{n}\|_{\infty}+\Dt_{n+1}\big[ C_{2}\|\ve^{n}\|_{\infty}+C_{1}(\Dt^{2}_{n+1}+h^{2})\big],
\err
where $C_{1}=\max\big\{C_{5}C_{8}+C_{6},~(C_{5}\Dt/2+1)\varepsilon^{2}L\|\phi\|_{L^{\infty}(0,T;W^{4,\infty}(\Omega))}/6\big\}$ and $C_{2}=C_{4}+C_{5}C_{7}$.
Summing up \eqref{eqnn7} from 0 to n, and together with the discrete the Gronwall's Lemma, the desired estimate \eqref{eqnn8} can be derived, which completes the proof.
 \end{proof}
\begin{theorem}\label{st_the}
Under the assumption of Theorem \ref{th3} and the additional condition $S_{2}\geq L^{2}/8$, the CN scheme \eqref{CN_2}  in the general mobility case satisfies
\bq\label{eq2_11}
\dps E_{h}(\Phi^{n+1})-E_{h}(\Phi^{n})\leq C\Dt_{n+1}(h^{4}+\Dt^{2}),
\eq
and consequently,
\bq\label{equn21}
\dps E_{h}(\Phi^{n+1})\leq E_{h}(\Phi^{0})+C T(h^{4}+\Dt^{2}), \quad 0\leq n\leq N-1,
\eq
for all $0\leq n\leq N-1$,  i.e., the discrete free energy is uniformly  bounded by the energy at the initial time plus a constant.
\end{theorem}
\begin{proof}
From $0<M_{0}\leq M(\rho)\leq L$ for $\rho\in[-1,1]$ and the definition of $\Lambda^{n+1/2}$ in \eqref{CN2_tsor2}, it follows that the matrix $\Lambda^{n+1/2}$ is invertible.
Moreover, we get
\bq\label{equnt4}
\tps(\Lambda^{n+1/2})^{-1}=\mbox{diag}\Big(\frac{1}{M(\vPhi^{n+1/2}_{1})},\frac{1}{M(\vPhi^{n+1/2}_{2})},\cdots,\frac{1}{M(\vPhi^{n+1/2}_{M^{2}})}\Big),
\eq
and
$1/L\leq1/M(\vPhi^{n+1/2}_{i})\leq 1/M_{0}$ for $i=1,2,\cdots, M^{2}$.
Therefore, we multiply both sides of \eqref{CN2_tsor2} with $(\Lambda^{n+1/2})^{-1}$ to obtain that
\bry
&\dps\big[\frac{1}{\Dt_{n+1}}+S_{2}\Dt_{n+1}\big](\Lambda^{n+\frac{1}{2}})^{-1}(\vPhi^{n+1}-\vPhi^{n})-\varepsilon^{2}D_{h}\frac{\vPhi^{n+1}+\vPhi^{n}}{2}+F'(\vPhi^{n+\frac{1}{2}})\nonumber\\
&\qquad\qquad\qquad\qquad\qquad\qquad\qquad\qquad\qquad\dps+S_{1}(\Lambda^{n+\frac{1}{2}})^{-1}\Big(\frac{\vPhi^{n+1}+\vPhi^{n}}{2}-\vPhi^{n+\frac{1}{2}}\Big)=0,
\ery
Furthermore, multiplying the above equality with $(\vPhi^{n+1}-\vPhi^{n})^{T}$, gives
\brr\label{equnt2}
&\dps -\frac{\varepsilon^{2}}{2}(\vPhi^{n+1}-\vPhi^{n})^{T}D_{h}(\vPhi^{n+1}+\vPhi^{n})+(\vPhi^{n+1}-\vPhi^{n})^{T}F'(\vPhi^{n+1})\\[5pt]
&\quad\quad=\dps (\vPhi^{n+1}-\vPhi^{n})^{T}(F'(\vPhi^{n+1})-F'(\vPhi^{n+\frac{1}{2}}))\\
&\quad\qquad\dps-\big[\frac{1}{\Dt_{n+1}}+S_{2}\Dt_{n+1}\big](\vPhi^{n+1}-\vPhi^{n})^{T}(\Lambda^{n+\frac{1}{2}})^{-1}(\vPhi^{n+1}-\vPhi^{n})\\
&\quad\qquad\dps-S_{1}(\vPhi^{n+1}-\vPhi^{n})^{T}(\Lambda^{n+\frac{1}{2}})^{-1}\big(\frac{\vPhi^{n+1}+\vPhi^{n}}{2}-\vPhi^{n+\frac{1}{2}}\big).
\err
Due to \eqref{equnt1}, we have
 \bq\label{equnt3}
 \sum_{i=1}^{M^{2}}(F(\vPhi^{n+1}_{i})-F(\vPhi^{n}_{i}))\leq\dps (\vPhi^{n+1}-\vPhi^{n})^{T}F'(\vPhi^{n+1})+\frac{1}{2}|\vPhi^{n+1}-\vPhi^{n}|^{2}.
 \eq
Using the fact $\max_{-1\leq\rho\leq1} F''(\rho)\leq 2$, the Cauchy-Schwarz inequality, and  the Young's inequality  together with \eqref{equnt4}, we deduce that
 \brr\label{equnt5}
&\dps(\vPhi^{n+1}-\vPhi^{n})^{T}(F'(\vPhi^{n+1})-F'(\vPhi^{n+\frac{1}{2}}))\\
&\quad\dps\leq 4\Dt_{n+1}L|\vPhi^{n+1}-\vPhi^{n+\frac{1}{2}}|^{2}+\frac{|\vPhi^{n+1}-\vPhi^{n}|^{2}}{4\Dt_{n+1}L},\\[5pt]
&\dps-\big[\frac{1}{\Dt_{n+1}}+S_{2}\Dt_{n+1}\big](\vPhi^{n+1}-\vPhi^{n})^{T}(\Lambda^{n+\frac{1}{2}})^{-1}(\vPhi^{n+1}-\vPhi^{n})\\
&\quad\dps\leq -\big[\frac{1}{\Dt_{n+1}}+S_{2}\Dt_{n+1}\big]\frac{|\vPhi^{n+1}-\vPhi^{n}|^{2}}{L},\\[5pt]
&\dps-S_{1}(\vPhi^{n+1}-\vPhi^{n})^{T}(\Lambda^{n+\frac{1}{2}})^{-1}\big(\frac{\vPhi^{n+1}+\vPhi^{n}}{2}-\vPhi^{n+\frac{1}{2}}\big)\\
&\quad\dps\leq\frac{S_{1}^{2}L\Dt_{n+1}}{M^{2}_{0}}\Big|\frac{\vPhi^{n+1}+\vPhi^{n}}{2}-\vPhi^{n+\frac{1}{2}}\Big|^{2}+\frac{|\vPhi^{n+1}-\vPhi^{n}|^{2}}{4\Dt_{n+1}L}.
\err
Combining \eqref{equnt2} with \eqref{equnt3} and \eqref{equnt5}, together with the definition of $E_{h}(\Phi^{n})$ in \eqref{dis_eg}, we obtain
\brr\label{equnt6}
 &E_{h}(\Phi^{n+1})-E_{h}(\Phi^{n})\\
 &\quad\tps=h^{2}\Big[-\frac{\varepsilon^{2}}{2}\big[(\vPhi^{n+1})^{T}D_{h}\vPhi^{n+1}-(\vPhi^{n})^{T}D_{h}\vPhi^{n}\big]+\sum_{i=1}^{M^{2}}(F(\vPhi^{n+1}_{i})-F(\vPhi^{n}_{i}))\Big]\\[5pt]
  &\quad\tps\leq h^{2}\big[-\frac{\varepsilon^{2}}{2}(\vPhi^{n+1}-\vPhi^{n})^{T}D_{h}(\vPhi^{n+1}+\vPhi^{n})\\[5pt]
  &\qquad+(\vPhi^{n+1}-\vPhi^{n})^{T}F'(\vPhi^{n+1})+\frac{1}{2}|\vPhi^{n+1}-\vPhi^{n}|^{2}\big]\\[5pt]
&\quad\leq\tps h^{2}\big[ 4\Dt_{n+1}L|\vPhi^{n+1}-\vPhi^{n+\frac{1}{2}}|^{2}+\frac{S^{2}_{1}L\Dt_{n+1}}{M^{2}_{0}}\Big|\frac{\vPhi^{n+1}+\vPhi^{n}}{2}-\vPhi^{n+\frac{1}{2}}\Big|^{2}\\[5pt]
 &\qquad\tps-\big(\frac{1}{2\Dt_{n+1}}+S_{2}\Dt_{n+1}-\frac{L}{2}\big)\frac{|\vPhi^{n+1}-\vPhi^{n}|^{2}}{L}\big].
\err
 Similar to derive \eqref{eqnt4}, we can use the triangle inequality to obtain that
 \brr\label{equnt7}
 &h^{2}|\vPhi^{n+1}-\vPhi^{n+\frac{1}{2}}|^{2}\\
 &\quad\leq\dps h^{2}\big[|\ve^{n+1}|^{2}+|\ve^{n+\frac{1}{2}}|^{2}\big]+\frac{\Dt^{2}_{n+1}|\Omega|}{4}\|\phi_{t}\|_{L^{\infty}(0,T;L^{\infty}(\Omega))},\\[5pt]
 &\quad\leq\dps|\Omega|\big[\|\ve^{n+1}\|^{2}_{\infty}+\|\ve^{n+\frac{1}{2}}\|^{2}_{\infty}\big]+\frac{\Dt^{2}_{n+1}|\Omega|}{4}\|\phi_{t}\|_{L^{\infty}(0,T;L^{\infty}(\Omega))},\\[5pt]
&\dps h^{2}\Big|\frac{\vPhi^{n+1}+\vPhi^{n}}{2}-\vPhi^{n+\frac{1}{2}}\Big|^{2}\\
&\quad\leq\dps h^{2}\big[|\ve^{n+1}|^{2}+|\ve^{n}|^{2}+|\ve^{n+\frac{1}{2}}|^{2}\big]+\frac{\Dt^{2}_{n+1}|\Omega|}{4}\|\phi_{t}\|_{L^{\infty}(0,T;L^{\infty}(\Omega))}\\[5pt]
 &\quad\leq\dps|\Omega|\big[\|\ve^{n+1}\|^{2}_{\infty}+\|\ve^{n}\|^{2}_{\infty}+\|\ve^{n+\frac{1}{2}}\|^{2}_{\infty}\big]+\frac{\Dt^{2}_{n+1}|\Omega|}{4}\|\phi_{t}\|_{L^{\infty}(0,T;L^{\infty}(\Omega))}.
 \err
Thus, we deduce from \eqref{equnt6}, \eqref{equnt7}, \eqref{eqnn8} and \eqref{eqn1_13} that
\bry
E_{h}(\Phi^{n+1})-E_{h}(\Phi^{n})&\leq\tps C_{9}\Dt_{n+1}\big(h^{4}+\Dt^{2}_{n+1}\big)-\big(\frac{1}{2\Dt_{n+1}}+S_{2}\Dt_{n+1}-\frac{L}{2}\big)\frac{|\vPhi^{n+1}-\vPhi^{n}|^{2}}{L}\\
&\leq\dps C_{9}\Dt_{n+1}\big(h^{4}+\Dt^{2}_{n+1}\big),
 \ery
where we have used
\beq
\frac{1}{2\Dt_{n+1}}+S_{2}\Dt_{n+1}-\frac{{\color{black}L}}{2}\geq 2\sqrt{\frac{S_{2}}{2}}-\frac{L}{2}\geq0.
\eeq
for any $S_2\geq L^{2}/8.$
 Summing the above inequality up from 0 to $n$ gives the desired result \eqref{equn21}.
\end{proof}
\section{Numerical experiments}
\setcounter{equation}{0}

In this section, some numerical experiments are presented to validate the theoretical results of the proposed CN scheme \eqref{CN_2} in terms of accuracy and preservation of the MBP.
 Throughout the numerical tests, the models are subject to the homogenous Neumann boundary condition, and the central finite difference method is exploited for the spatial discretization.

\subsection{Temporal convergence}
We consider the Allen-Cahn equation \eqref{prob} with the parameter $\varepsilon=0.01$, the initial value
$$\phi_0(x,y)=0.1(\cos3x\cos2y+\cos5x\cos5y),$$
and two types of mobility functions: one is the constant mobility  $M(\phi)\equiv1$ and the other is the nonlinear degenerate mobility $M(\phi)=1-\phi^{2}$.
The computational domain is set to be $\Omega = (0,1)^2$ and the terminal time is chosen to be $T=1$.
We  fix the uniform spatial mesh size $h=1/1024$ which is small enough so that the spacial discretization error is negligible compared to that by the temporal discretization.
The stabilizing parameters are chosen to be $S_{1}=S_{2}=2$.
Due to no analytical solution available for this numerical experiment, we evaluate the numerical solution errors in the discrete $L^{\infty}$ and $H^1$ norms, respectively, as follows:
\beq
e^T_{\infty}(N)=\|\Phi^{N}-\Phi^{2N}\|_{\infty}, \quad e^T_{H^1}(N)=\|\Phi^{N}-\Phi^{2N}\|_{H^1_h},
\eeq
where $N$ denotes the number of subintervals for the time domain $[0,1]$ and $\Phi^{N}$ is the corresponding numerical solution at the terminal time $T=1$.

Firstly, we test the convergence rate of the CN scheme \eqref{CN_2} on the uniform temporal meshes with time step size $\tau$ ranging
from $1/10$ to $1/320$ (i.e., $N$ changes from $10$ to $320$).
In Fig. \ref{fig1}, we present the $H^{1}$ and $L^{\infty}$ errors at $T=1$ as functions of the time step sizes in log-log scale.
It is shown that the CN scheme \eqref{CN_2} achieves the expected second-order temporal accuracy for both mobility cases.
Next, we numerically investigate the error behaviors of the CN scheme \eqref{CN_2} with a sequence of nonuniform temporal meshes, which is produced  by $40\%$ perturbation of the uniform ones $\{t_{n}=n/N\}_{n=0}^{N}$.
We denote by $\big\{\gamma_{n}=\Dt_{n}/\Dt_{n-1}\big\}_{n=2}^{N}$ the adjacent time-step ratios.
Once again the observed error behaviors in Table \ref{table1} achieve the desired second order accuracy in time for all cases.

\begin{figure*}[!htbp]
\centerline{\includegraphics[scale=0.5]{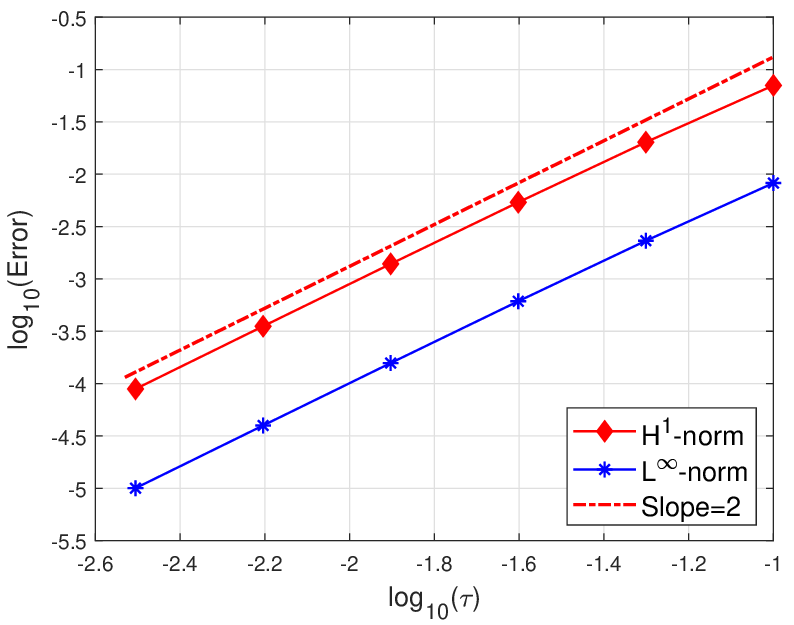}\includegraphics[scale=0.5]{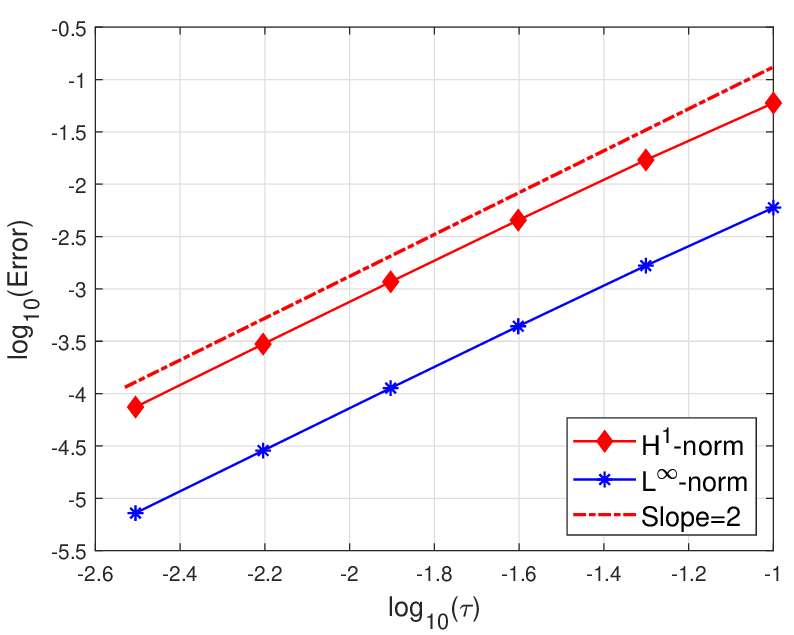}}
\caption{The error behavior with respect to the time step size for the CN scheme \eqref{CN_2} with uniform time steps.
Left: $M(\phi)\equiv1$; right: $M(\phi)=1-\phi^{2}$.
}\label{fig1}
\end{figure*}

\begin{table}[!htbp]\small
  \begin{center}
      \caption{Numerical solution errors and convergence rates of the CN scheme \eqref{CN_2} with nonuniform time steps.}\label{table1}
   \scalebox{0.86}{ \begin{tabular}{|ccc|cc|cc|cc|cc|} \hline
    \multicolumn{3}{|c|}{Time steps}&\multicolumn{4}{|c|}{$M(\phi)= 1$ }&\multicolumn{4}{|c|}{$M(\phi)=1-\phi^{2}$ } \\ \hline
     $N$&$\tau$&$\max\{\gamma_{n}\}$&$e^T_{H^1_h}$&Order&$e^T_{\infty}$&Order&$e^T_{H^{1}_{h}}$&Order&$e^T_{\infty}$&Order\\ \hline
     $10$ &1.629e-1&4.322&9.236e-2&--   &1.090e-2&--    &7.823e-2&--   &7.927e-3&--   \\
     $20$ &8.254e-2&4.660&2.768e-2&1.77&3.198e-3&1.80 &2.334e-2&1.78&2.310e-3&1.81 \\
     $40$ &3.953e-2&4.950&6.721e-3&1.92&7.659e-4&1.94 &5.657e-3&1.93&5.511e-4&1.95 \\
    $80$  &2.105e-2&6.901&1.933e-3&1.98&2.184e-4&1.99 &1.625e-3&1.98&1.569e-4&1.99 \\
    $160$ &1.075e-2&6.997&4.562e-4&2.15&5.141e-5&2.15 &3.838e-4&2.15&3.689e-5&2.15 \\
    $320$ &5.546e-3&8.084&1.119e-4&2.12&1.261e-5&2.12 &9.517e-5&2.11&8.980e-6&2.14 \\
      \hline
    \end{tabular}}
  \end{center}
\end{table}

\subsection{MBP preservation}

To test the MBP preservation of the proposed CN scheme \eqref{CN_2}, we consider two well-known benchmark problems governed by the
Allen-Cahn equations. One is the grain coarsening dynamic process  and the other is the shrinking bubble problem \cite{Chen98}.

\paragraph*{\bf The grain coarsening dynamics}

In this numerical experiment, we investigate the coarsening dynamics governed by the Allen-Cahn equation \eqref{prob} with a nonlinear degenerate mobility $M(\phi)=1-\phi^{2}$ and a random initial data ranging from $-0.1$ to $0.1$.
The domain is set to be $\Omega=(0,1)^{2}$ with the width parameter $\varepsilon=1/256\approx3.91e$-$3$, and the uniform spatial mesh with $h=1/256$ is used for spatial discretization.
In particular, for such a nonlinear mobility function, it is of essential importance to preserve the numerical {\color{black}solution $\phi\in[-1,1]$} in the numerical algorithm. Otherwise, it may lead to the ill-posedness of the proposed numerical scheme, and the numerical solution blowing up during the time simulation, as shown in Fig. \ref{fig2} (a).

The MBP preservation of the proposed CN scheme  \eqref{CN_2} is investigated for this example  through a long time simulation up to $T=3000$ using the CN scheme \eqref{CN_2} with several large  time step sizes.
We set $S_{1}=4/5$ satisfying the condition \eqref{eqn1_3}, and  choose two  values for the stabilizing parameter $S_{2}$:
 one is the case of $S_{0}=0$ leading to the conditional MBP preservation of the CN scheme,  the other is $S_{2}=\big(\frac{S_{1}}{4}+\frac{L\varepsilon^{2}}{h^{2}}\big)^{2}$ satisfying \eqref{res_s} to guarantee the unconditional MBP preservation of the CN scheme  \eqref{CN_2} as stated in Theorem \ref{thmMBP}.
 The evolutions of the supremum norm and energy of the simulated solutions for both cases are presented in Fig. \ref{fig2}.
 We observe that the CN scheme \eqref{CN_2} with $S_{2}=0$ and $\Dt=2$ preserve the MBP, and the numerical solutions blow up at a finite time about $t=30$, see \ref{fig2}-(a).
It is also seen in Fig. \ref{fig2}-(b) that the CN scheme preserve the MBP under all tested time step sizes when the two stabilizing parameters $S_{1}$ and $S_{2}$ satisfy the conditions \eqref{eqn1_3} and \eqref{res_s}, respectively; furthermore, it also achieves energy dissipation for all cases in the sense of $E_{h}(\phi^{n+1})\leq E_{h}(\phi^{n})$ for $n=0,\cdots,N-1$.

One of main advantages of the unconditionally stable schemes is that it can be easily adopted by an adaptive time strategy.
This is particularly useful for the long time simulation of the coarsening dynamic process, in which the phase transition usually goes through several different stages within a long period:  changes quickly at the beginning and then rather slowly until it reaches a steady state.
{\color{black} There already exist several efficient time-adaptive strategies \cite{HQ23,HQ22,LTZ20,Shen17_2,QZT11,STY16} that can be used in conjunction with numerical schemes with variable time steps.}
We will exploit the use of the following robust time adaptive strategy with the CN scheme \eqref{CN_2} in the simulation, which is based on the energy
variation  \cite{QZT11} to efficiently  simulate the coarsening dynamic process:
\bq\label{adp}
\dps\Dt_{n+1}=\max\big(\Dt_{min},\frac{\Dt_{max}}{\sqrt{1+\alpha |E^{'}(t)|^{2}}}\big),
\eq
where $\Dt_{min},\Dt_{max}$ are the predetermined minimum and maximum time step sizes, and $\alpha$ is a positive constant parameter.
Obviously according to this type of time strategy, the numerical scheme will automatically select large time step sizes when energy variation is big and a small ones otherwise.
 The parameters are set to be $\tau_{min}=10^{-5}$, $\tau_{max}=0.1$, and $\alpha=10^{5}$ in this test.
In the simulation, we choose $S_{1}=4/5$ and $S_{2}=\big(\frac{S_{1}}{4}+\frac{L\varepsilon^{2}}{h^{2}}\big)^{2}$ such that both of the conditions \eqref{eqn1_3} and \eqref{res_s} are satisfied.
As shown in Fig. \ref{fig2_2}-(a), the CN scheme always preserves the MBP property during the simulation up to $T=150000.$
Moreover, we also display several snapshots of the simulated phase structures (around the times $t = 500, 1000, 5000, 10000, 13000, 15000$) in Fig. \ref{fig2_1},
and the evolutions of the energy and the adaptive time step sizes in Fig. \ref{fig2_2}-(b)\&(c).
As shown in Fig. \ref{fig2_2}-(b)\&(c), the computed energy is always dissipative in time and the adaptive time stepping strategy scheme automatically select the time step sizes according the changing rate of the free energy, which demonstrates the efficiency of our proposed scheme adopted with the time adaptive strategy \eqref{adp}.

\begin{figure*}[!htbp]
\centerline{\includegraphics[scale=0.38]{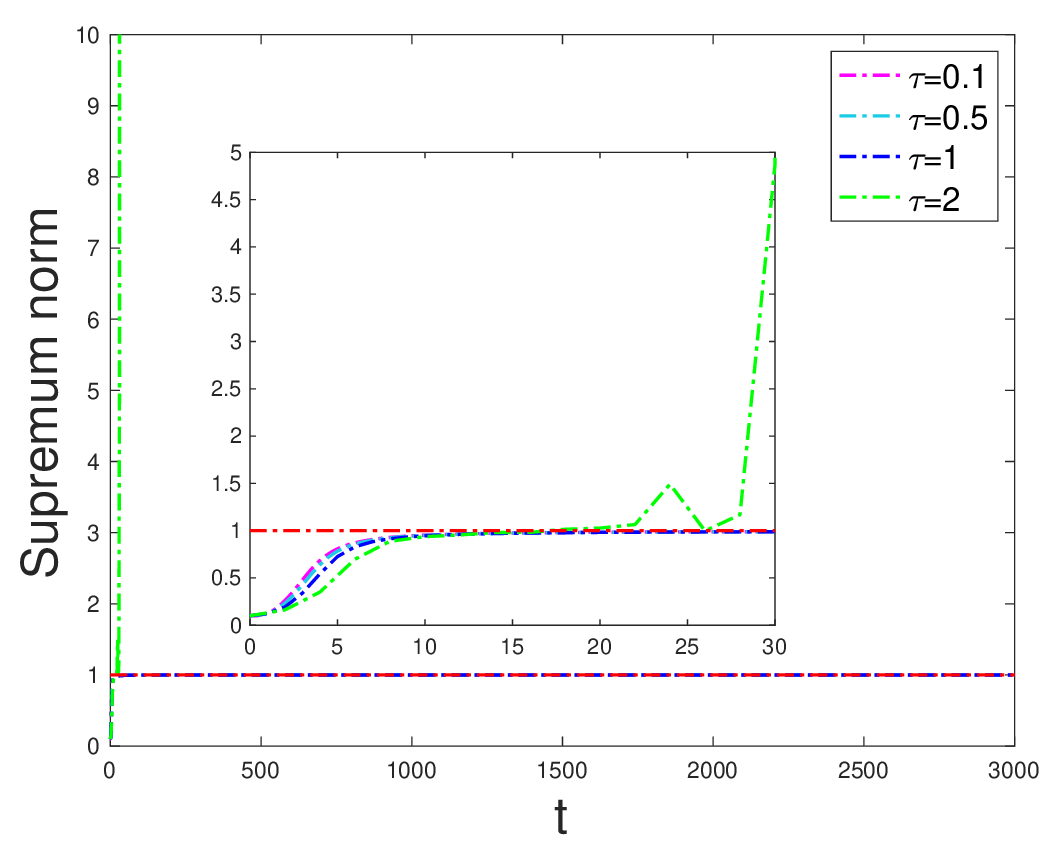}\includegraphics[scale=0.38]{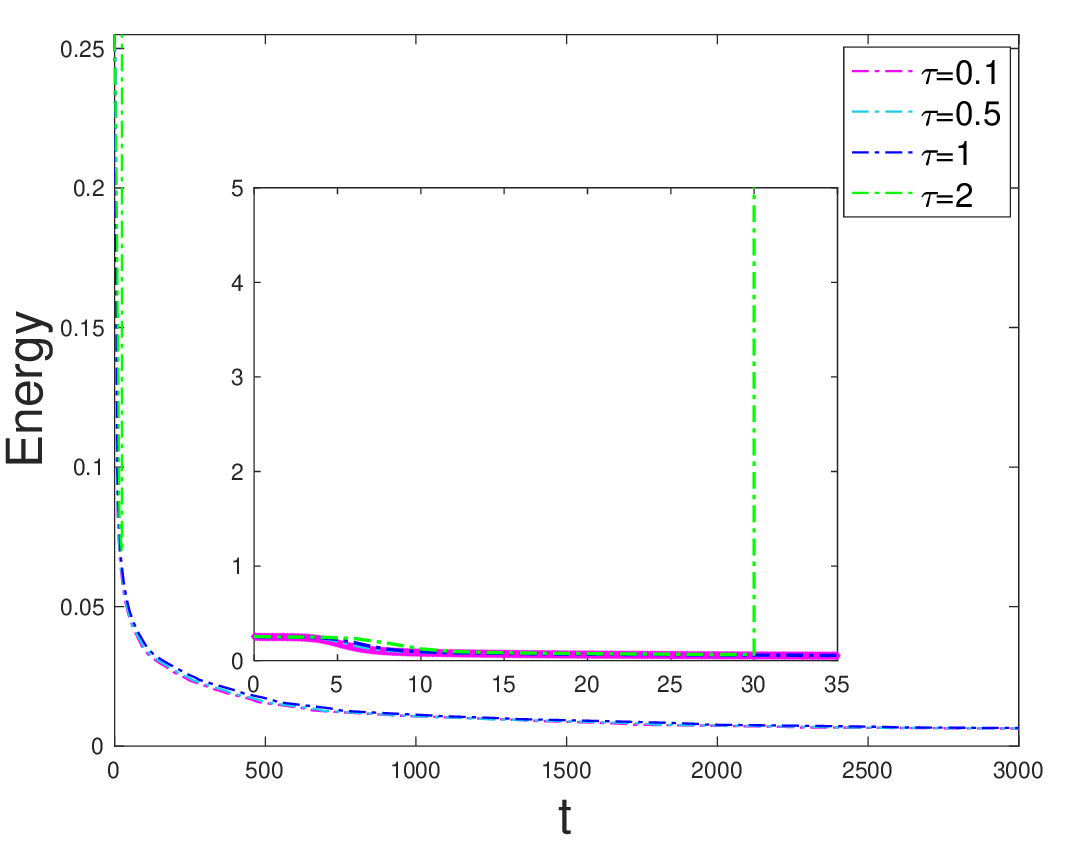}}
\centerline{(a) $S_{1}=\frac{4}{5}$ and $S_{2}=0$  }
\centerline{\includegraphics[scale=0.38]{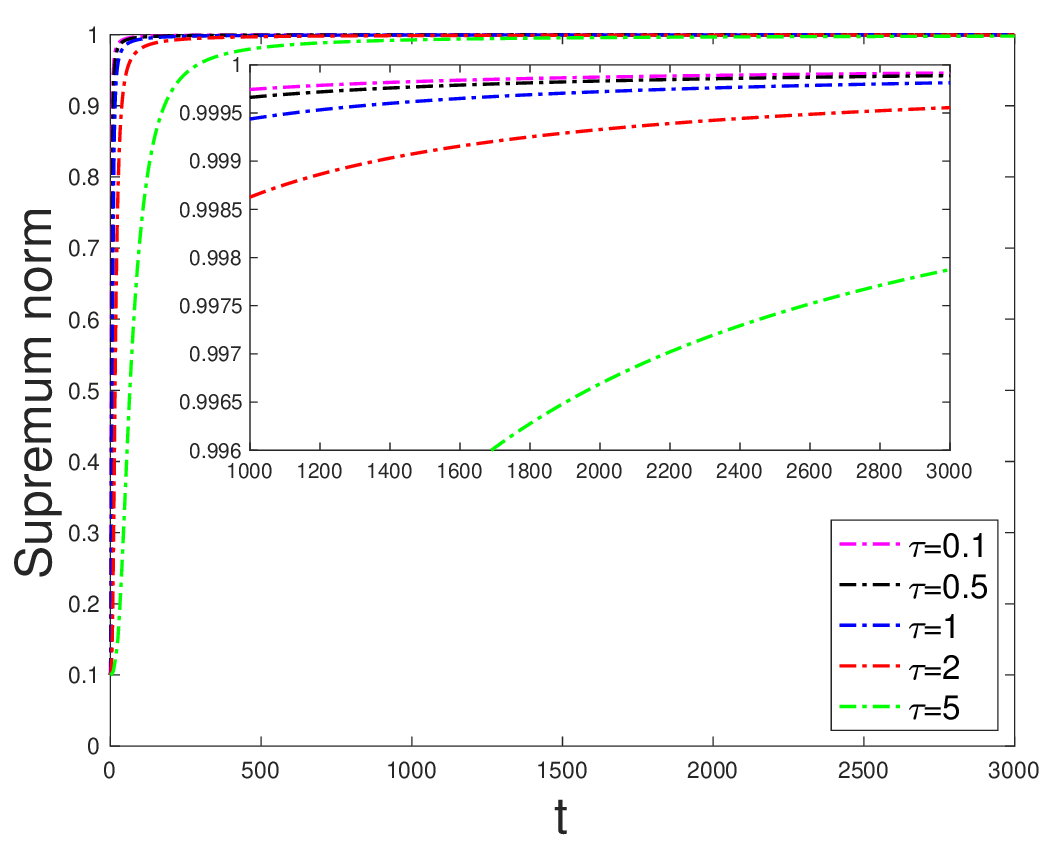}\includegraphics[scale=0.38]{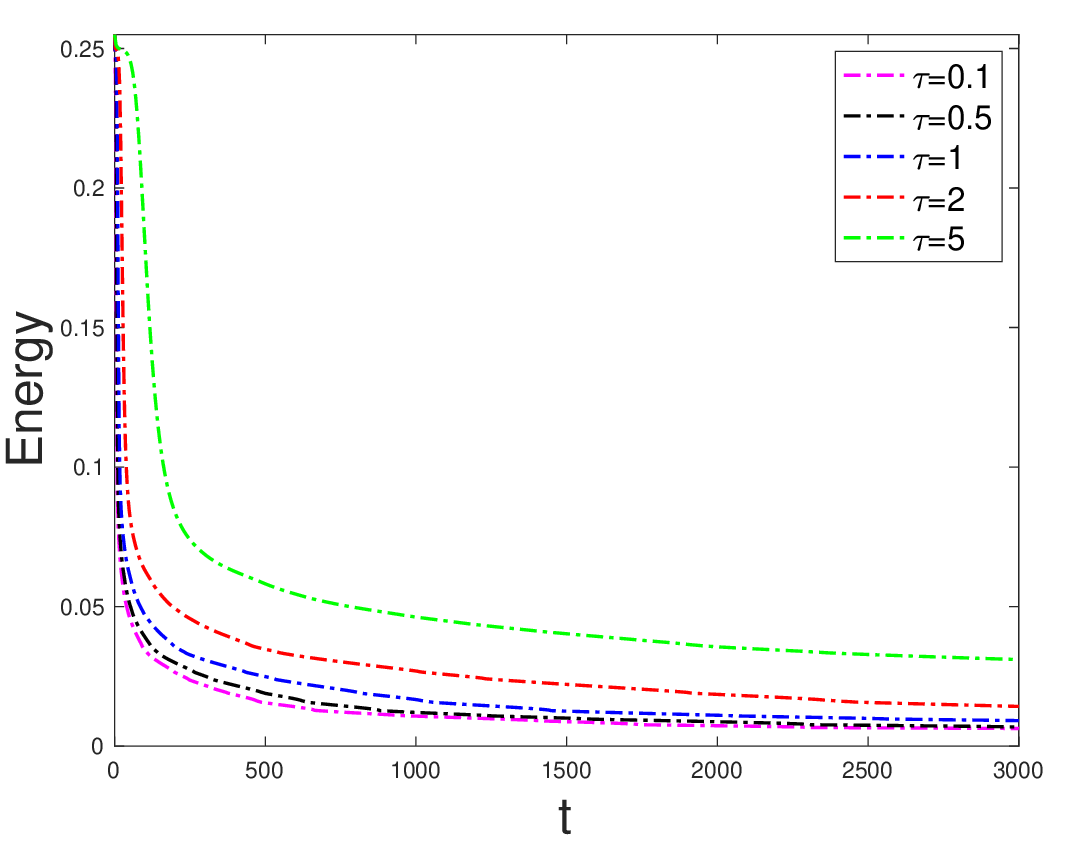}}
\centerline{(b) $S_{1}=\frac{4}{5}$ and $S_{2}=\big(\frac{S_{1}}{4}+\frac{L\varepsilon^{2}}{h^{2}}\big)^{2}$ }
\caption{The evolutions in time of the supremum norm (left) and the energy (right)  of the simulated solution produced by the CN scheme \eqref{CN_2} with some uniform time steps for the grain coarsening  problem with the  degenerate mobility.}\label{fig2}
\end{figure*}

\begin{figure*}[!htbp]
\centerline{\includegraphics[scale=0.3]{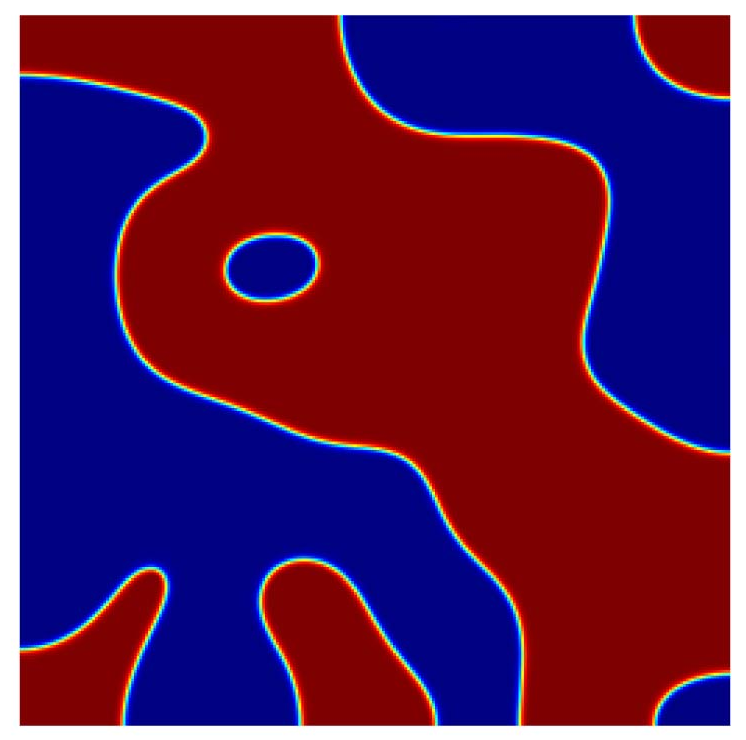}\includegraphics[scale=0.3]{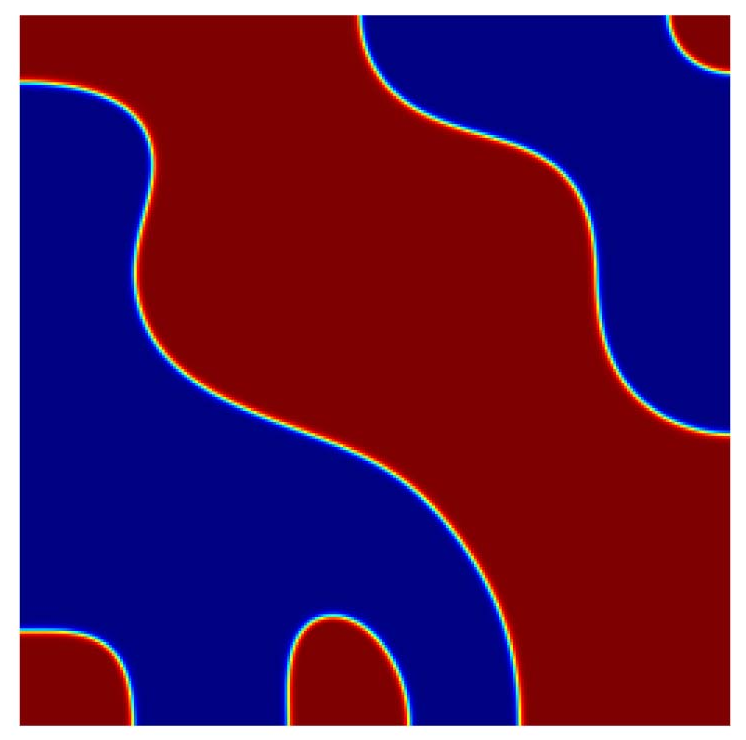}\includegraphics[scale=0.3]{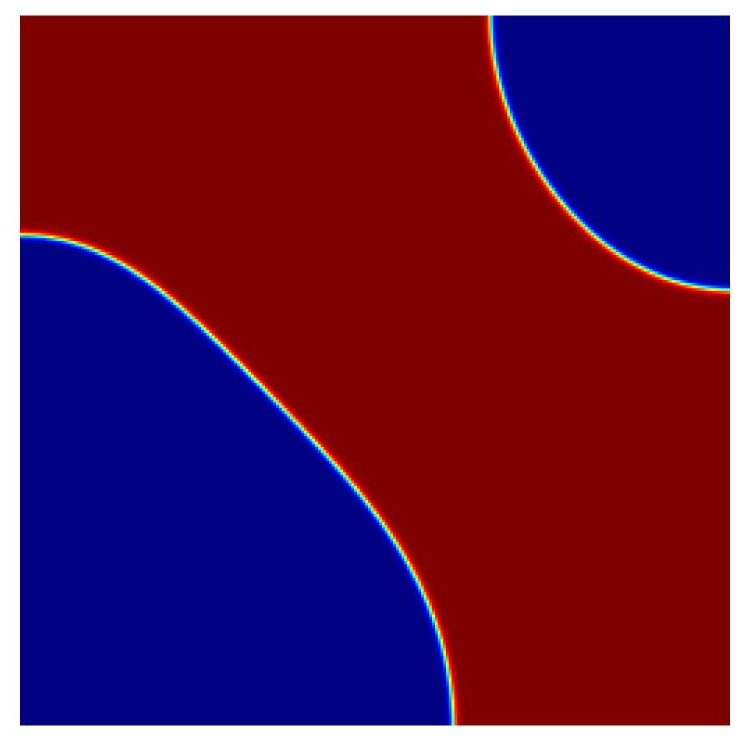}}
\centerline{\includegraphics[scale=0.3]{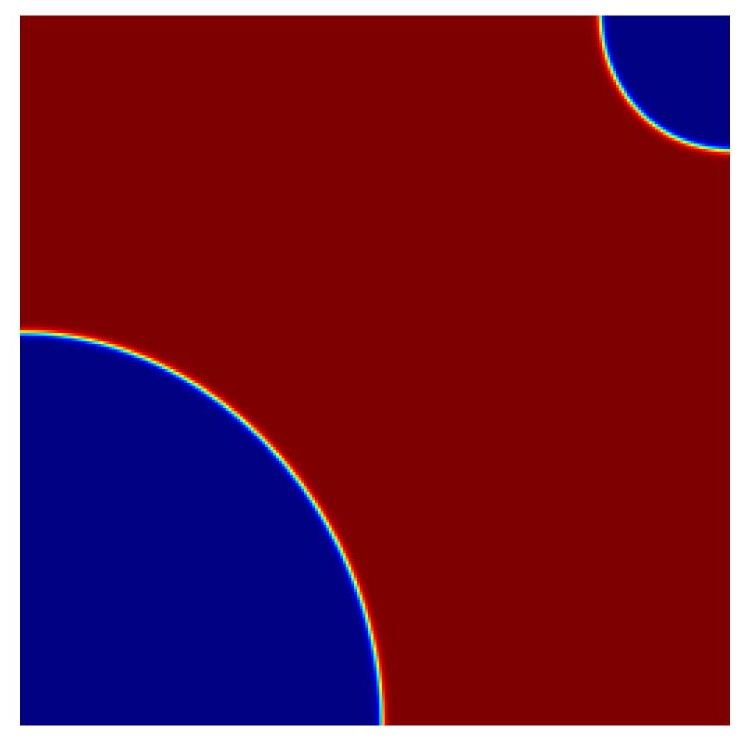}\includegraphics[scale=0.3]{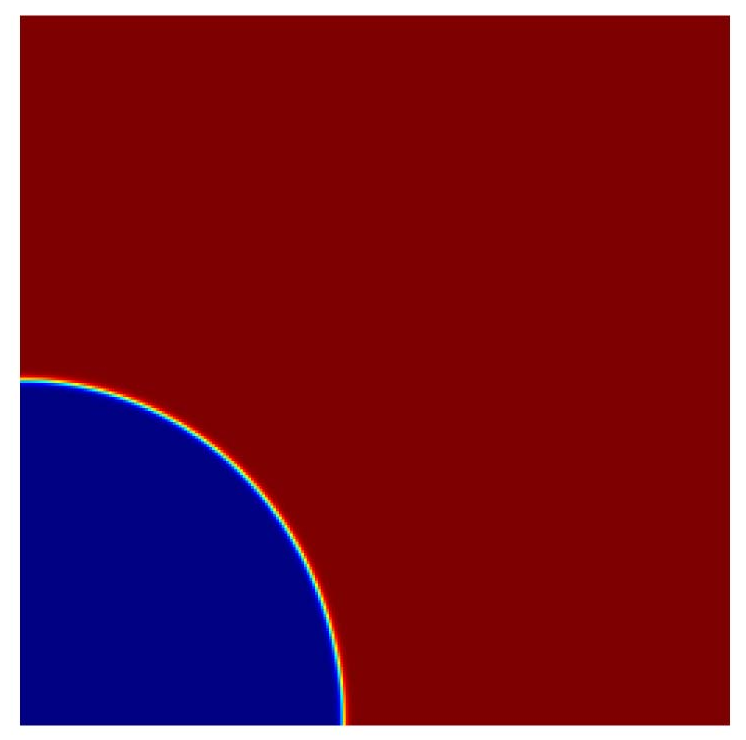}\includegraphics[scale=0.3]{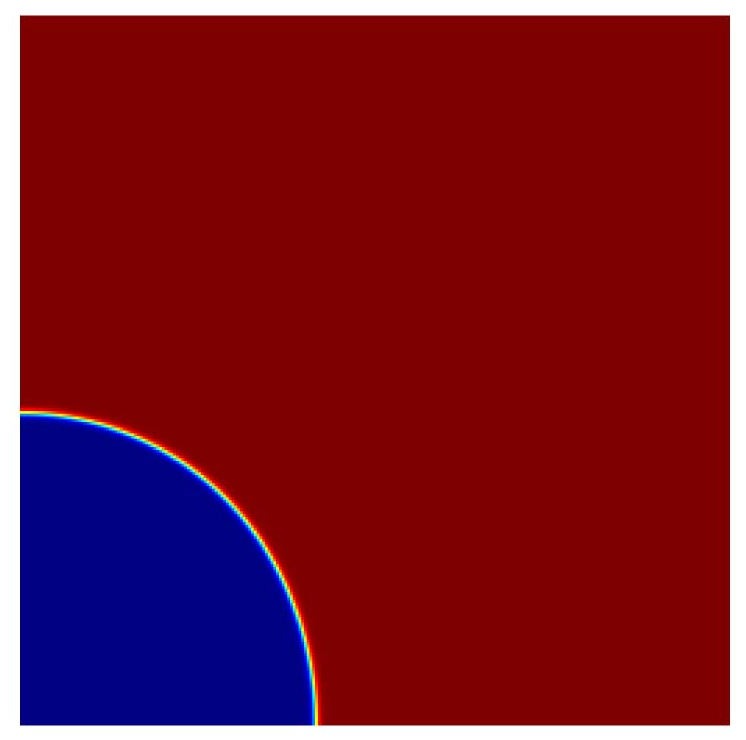}}
\caption{Snapshots of the simulated phase structures around the times $t=500$, $1000$, $5000$, $10000$, $13000$, and $15000$ produced by the CN scheme \eqref{CN_2} with the time adaptive strategy \eqref{adp} for the grain coarsening  problem with the  degenerate mobility.}
\label{fig2_1}
\end{figure*}

\begin{figure*}[!htbp]
\begin{minipage}[t]{0.325\linewidth}
\centerline{\includegraphics[scale=0.26]{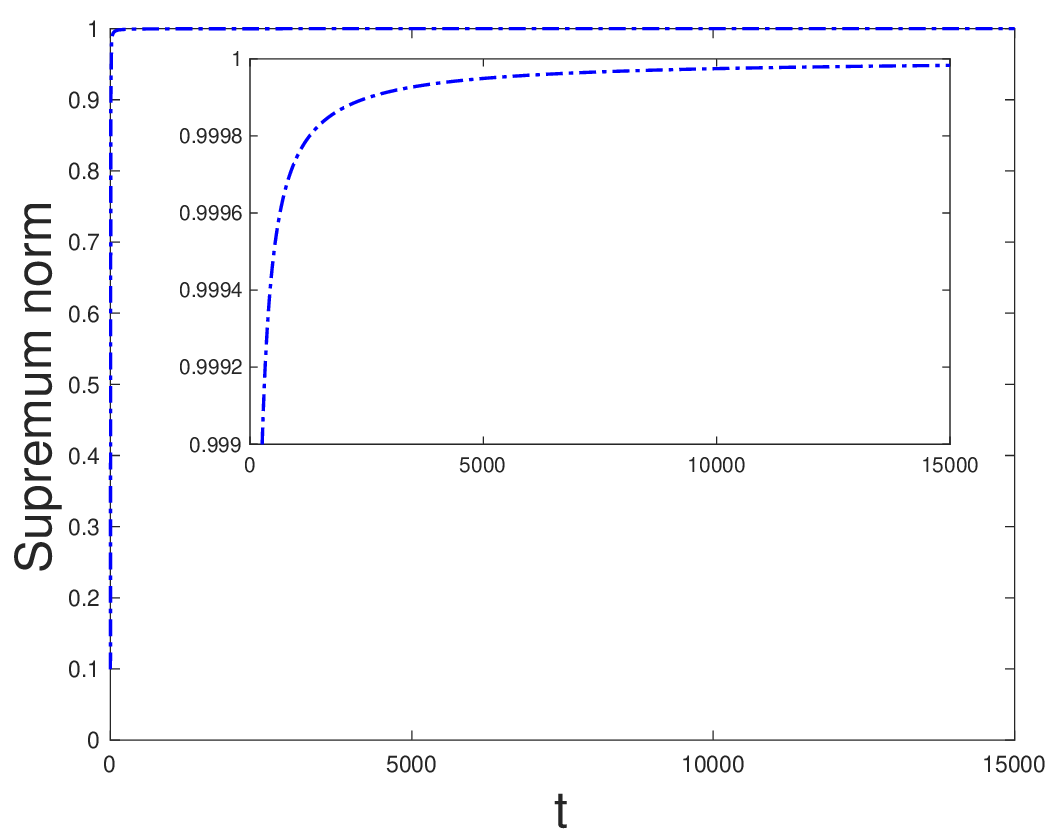}}
\centerline{(a) the supremum norm}
\end{minipage}
\begin{minipage}[t]{0.325\linewidth}
\centerline{\includegraphics[scale=0.26]{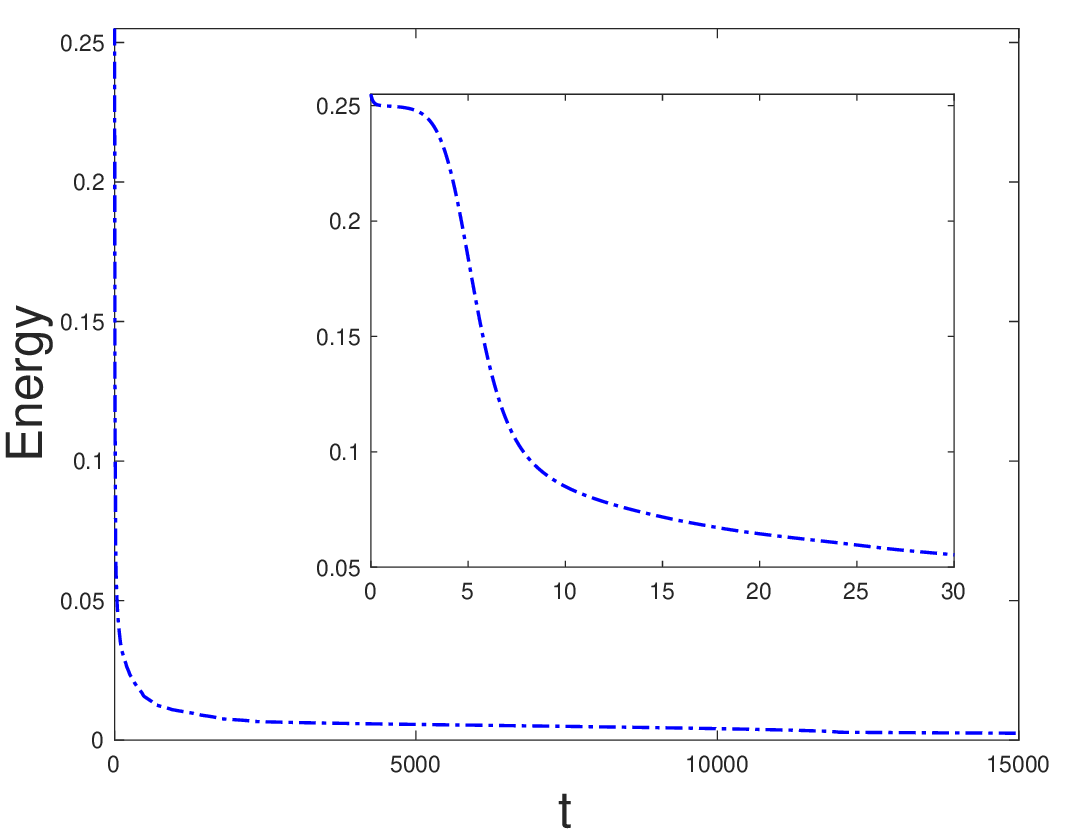}}
\centerline{(b) the energy }
\end{minipage}
\begin{minipage}[t]{0.325\linewidth}
\centerline{\includegraphics[scale=0.26]{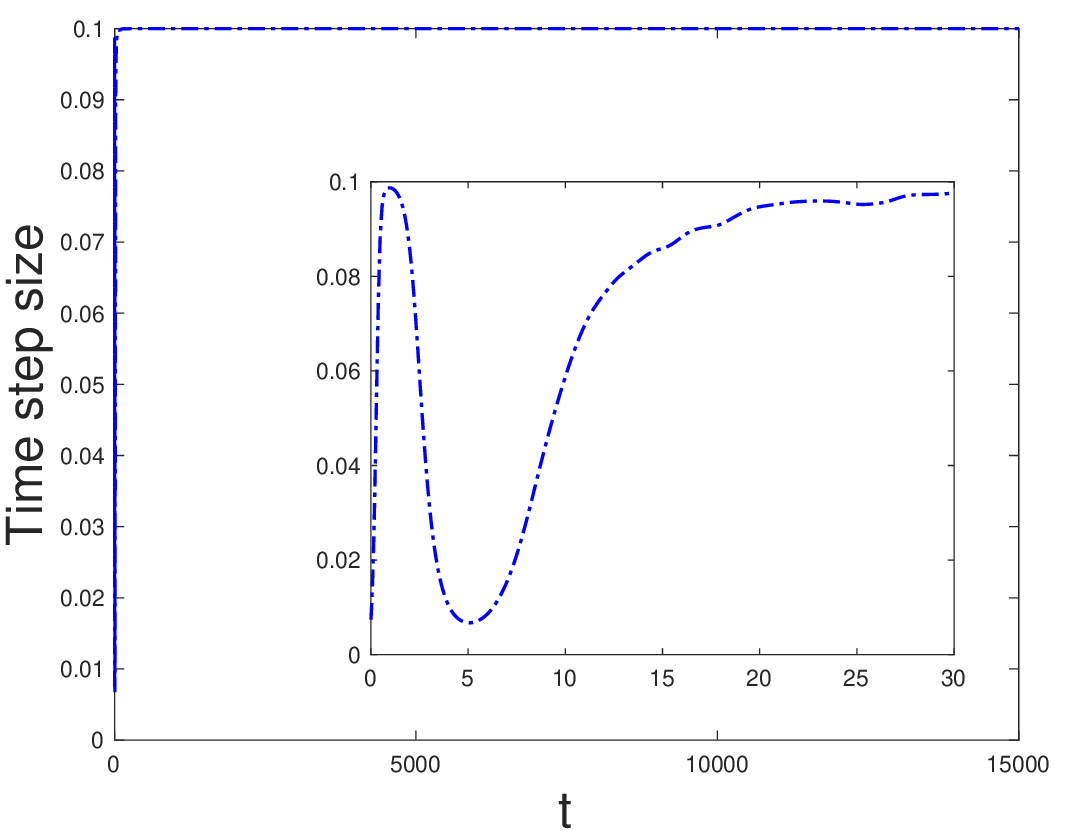}}
\centerline{(c) the time step size }
\end{minipage}
\caption{The evolutions in time of the supremum norm, the energy, and the time step sizes produced by the CN scheme \eqref{CN_2} with the time adaptive strategy \eqref{adp}   for the grain coarsening  problem  \eqref{adp}  with the degenerate mobility.}
\label{fig2_2}
\end{figure*}

\paragraph*{\bf The shrinking bubble problem}
We next use the shrinking bubble problem \cite{Chen98} to test the performance of the proposed CN scheme \eqref{CN_2} again with the time adaptive strategy \eqref{adp}.
The shrinking bubble problem is driven by the Allen-Cahn equation \eqref{prob} with $M(\phi)\equiv1$ and $\varepsilon=0.01$ in a rectangular domain $(-0.5,0.5)^2$. The initial bubble is a sphere of radius $R_{0}=0.2$ located at the center of the computational domain, given by
\beq
\phi_{0}(\x)=\begin{cases}
\begin{array}{r@{}l}
1,&\quad |\x|^2<0.2^{2},\\[1pt]
-1,&\quad |\x|^2\geq0.2^{2}.
\end{array}
\end{cases}
\eeq
As stated in \cite{JZZD15,HAX19,CGJMP19}, such a bubble is not stable and will shrink and finally disappear due the interface driving force.
Moreover, with assumption of a sufficient small $\varepsilon$, the radius of the circle at time t can be approximately expressed as follows
\bq\label{test_1}
R(t)=\sqrt{R^{2}_{0}-2\varepsilon^{2}t}.
\eq

We perform the simulation by combining the CN scheme \eqref{CN_2} with $h=1/512$ and the time adaptive strategy \eqref{adp} with $\tau_{min}=10^{-5}$, $\tau_{max}=0.01$, and $\alpha=10^{5}$.
Several snapshots of the computed bubble at the times $t =0,20,80,120,180, 200$ are plotted in {\color{black}Fig. \ref{fig3_1}}, showing that the bubble disappears at $t=200$ as predicted.
Moreover, it is observed in Fig. \ref{fig3} (a) that the simulated radius of the bubble is monotonously decreasing with a rate almost identical to the theoretical prediction \eqref{test_1}.
Furthermore, we present the evolution of the supremum norm of the numerical solutions  along with the time in Fig. \ref{fig3}-(b), which shows the MBP preservation of the proposed CN scheme \eqref{CN_2} during the whole simulation.
In the last line of Fig. \ref{fig3}, we plot the evolutions of the energy and the adaptive time step sizes to numerically demonstrate  the energy dissipation and  the efficiency of the proposed CN scheme \eqref{CN_2} adopted with the time adaptive strategy \eqref{adp}.

\begin{figure}[!htbp]
\centerline{\includegraphics[scale=0.3]{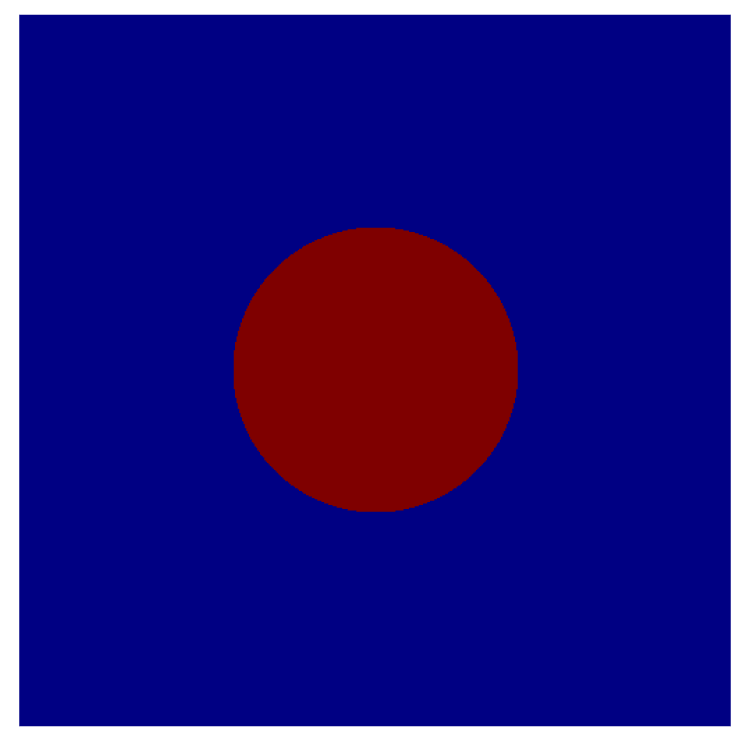}\includegraphics[scale=0.3]{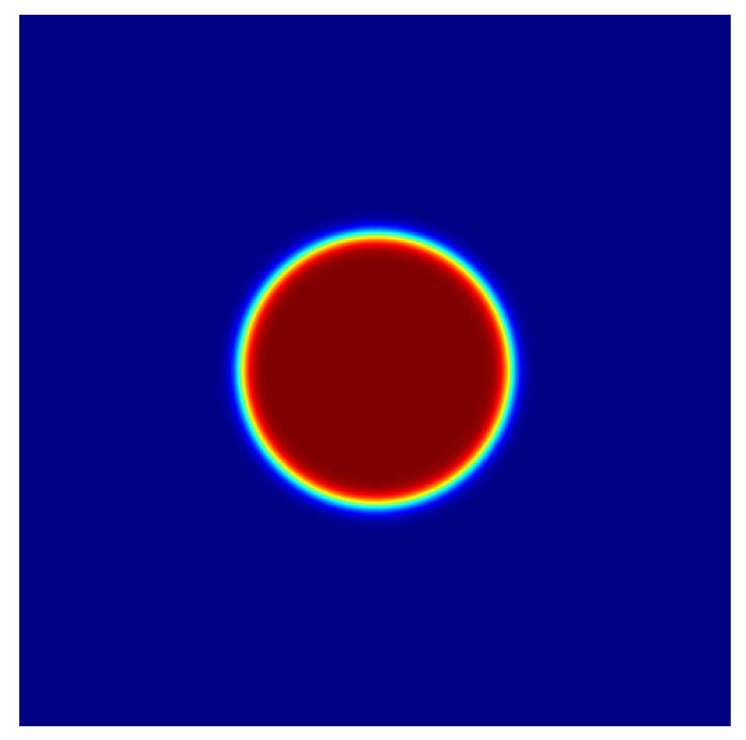}\includegraphics[scale=0.3]{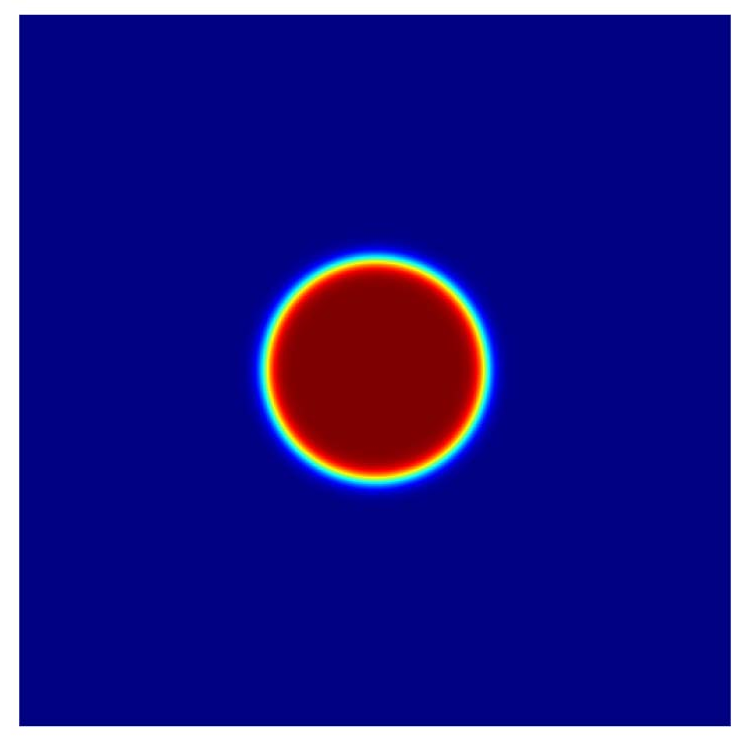}}
\centerline{\includegraphics[scale=0.3]{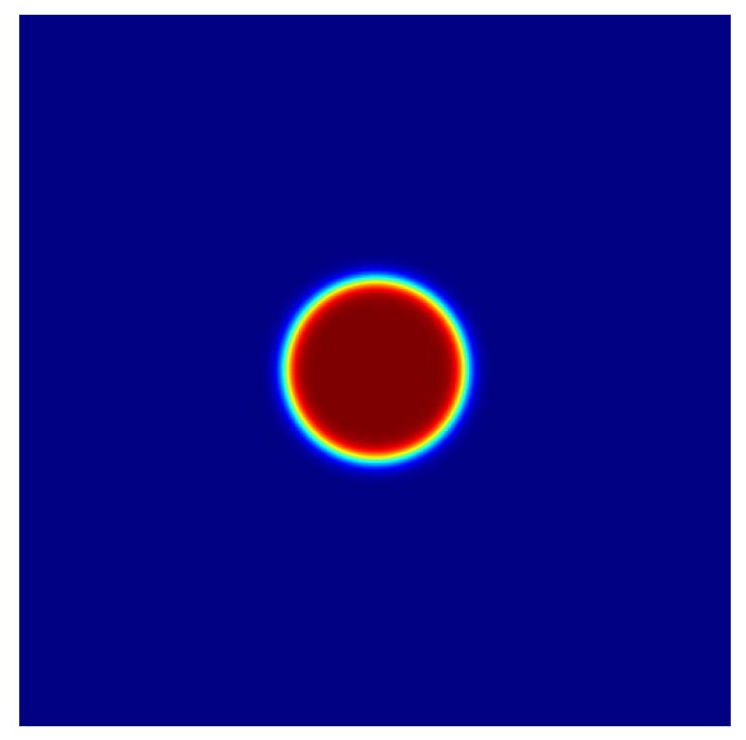}\includegraphics[scale=0.3]{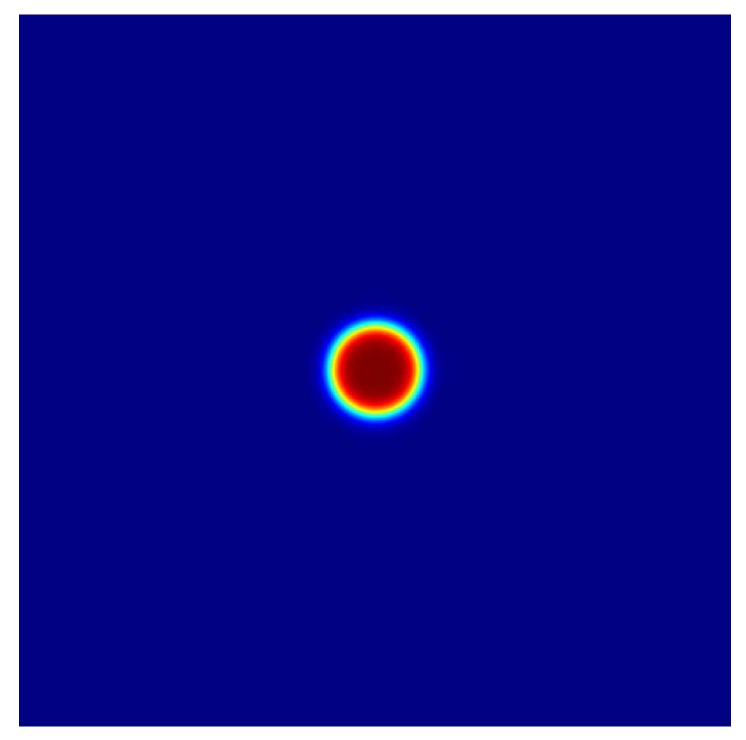}\includegraphics[scale=0.3]{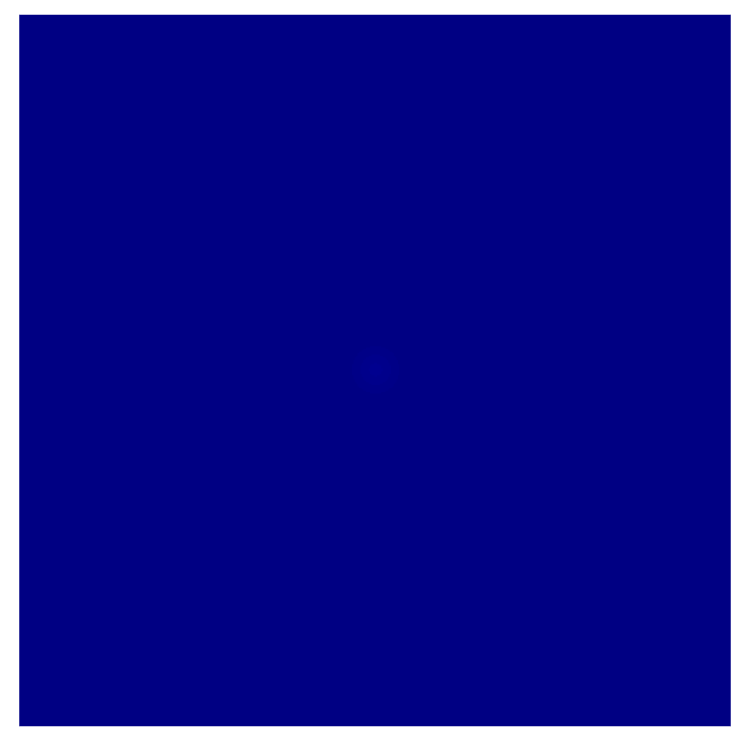}}
\caption{Snapshots of the simulated phase structures at the times $t = 0$, $20$, $80$, $120$, $180$,  and $200$ produced by the CN scheme \eqref{CN_2} for with the time adaptive strategy \eqref{adp}  for the shrinking bubble problem.}
\label{fig3_1}
\end{figure}

\begin{figure*}[!htbp]
\begin{minipage}[t]{0.45\linewidth}
\centerline{\includegraphics[scale=0.36]{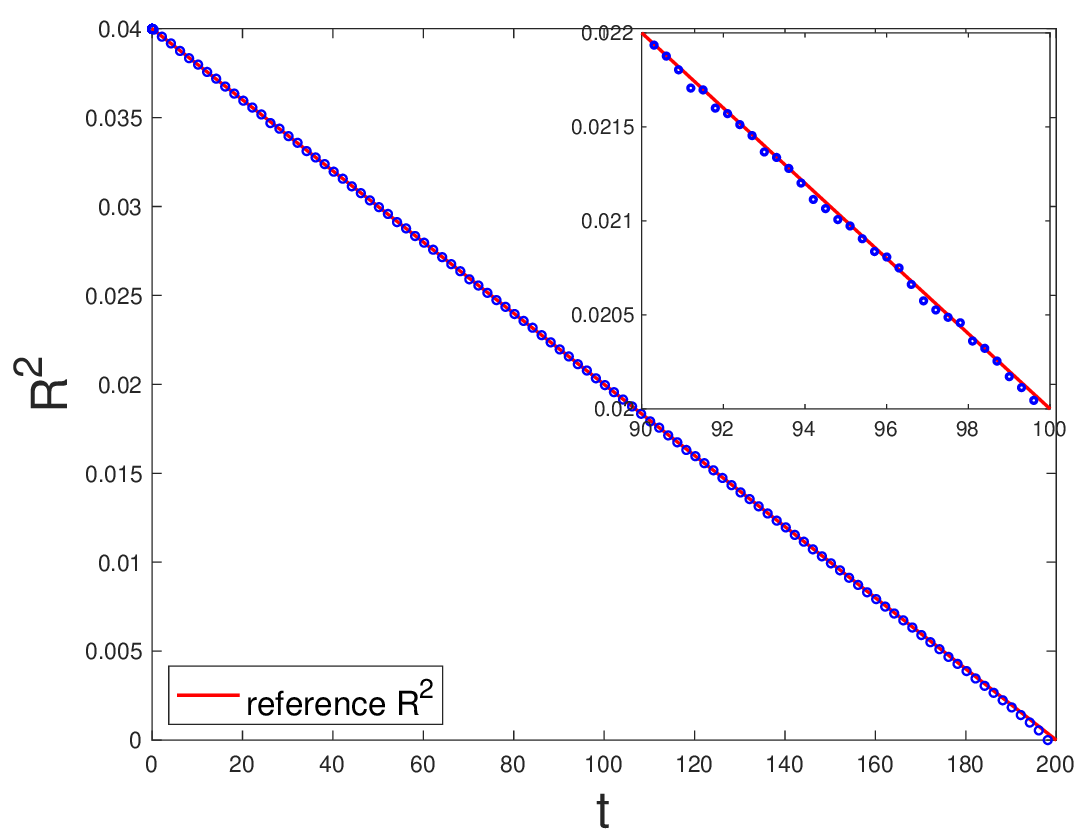}}
\centerline{(a) the radius  }
\end{minipage}
\begin{minipage}[t]{0.45\linewidth}
\centerline{\includegraphics[scale=0.36]{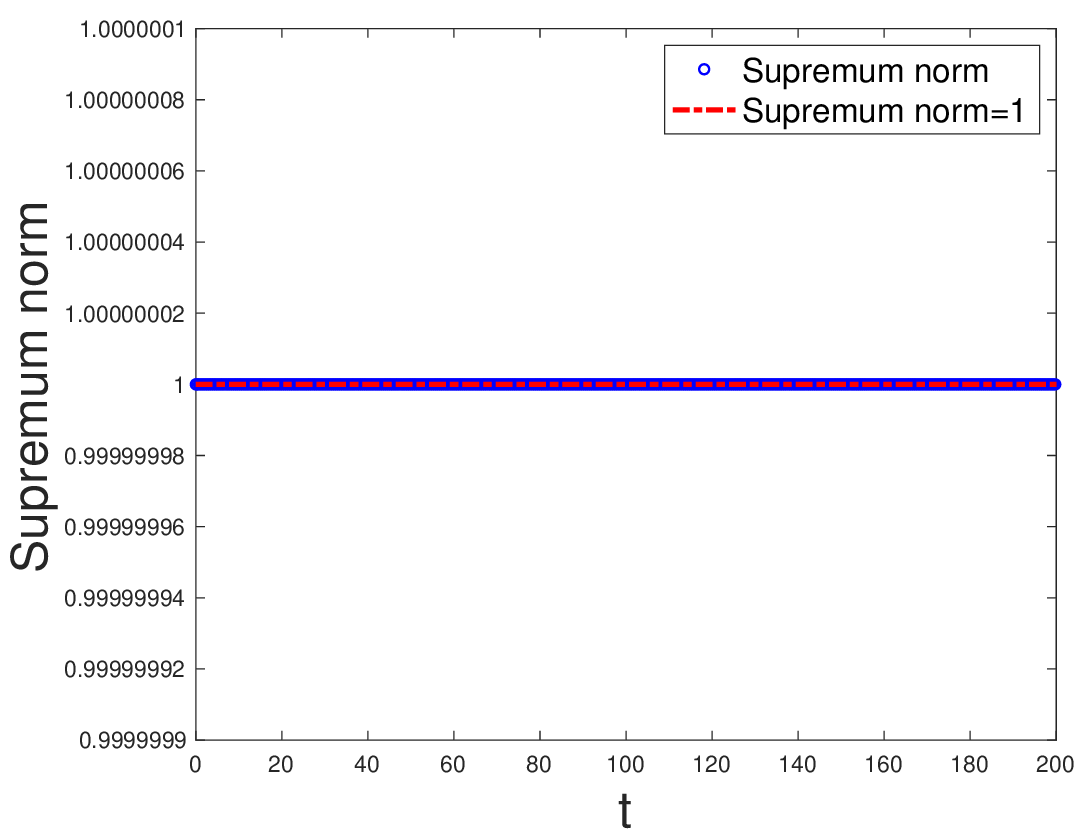}}
\centerline{(b) the supremum norm}
\end{minipage}
\vskip 3mm
\begin{minipage}[t]{0.45\linewidth}
\centerline{\includegraphics[scale=0.36]{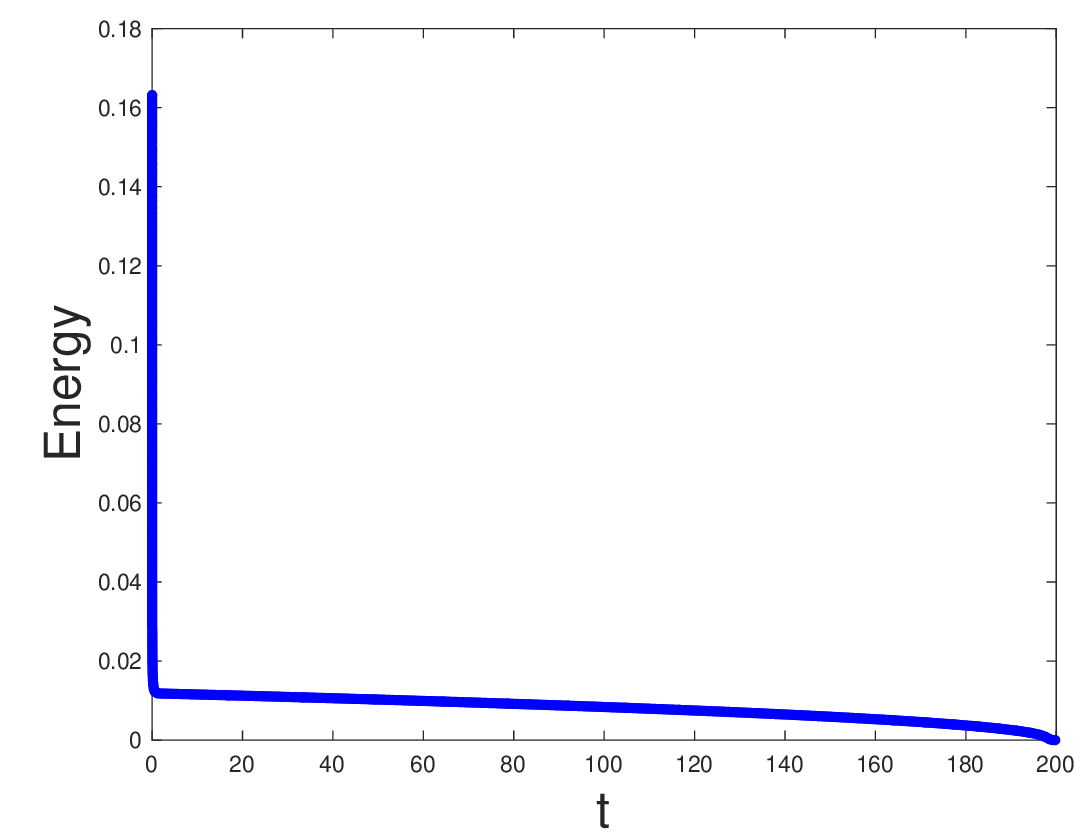}}
\centerline{(c) the energy}
\end{minipage}
\begin{minipage}[t]{0.45\linewidth}
\centerline{\includegraphics[scale=0.36]{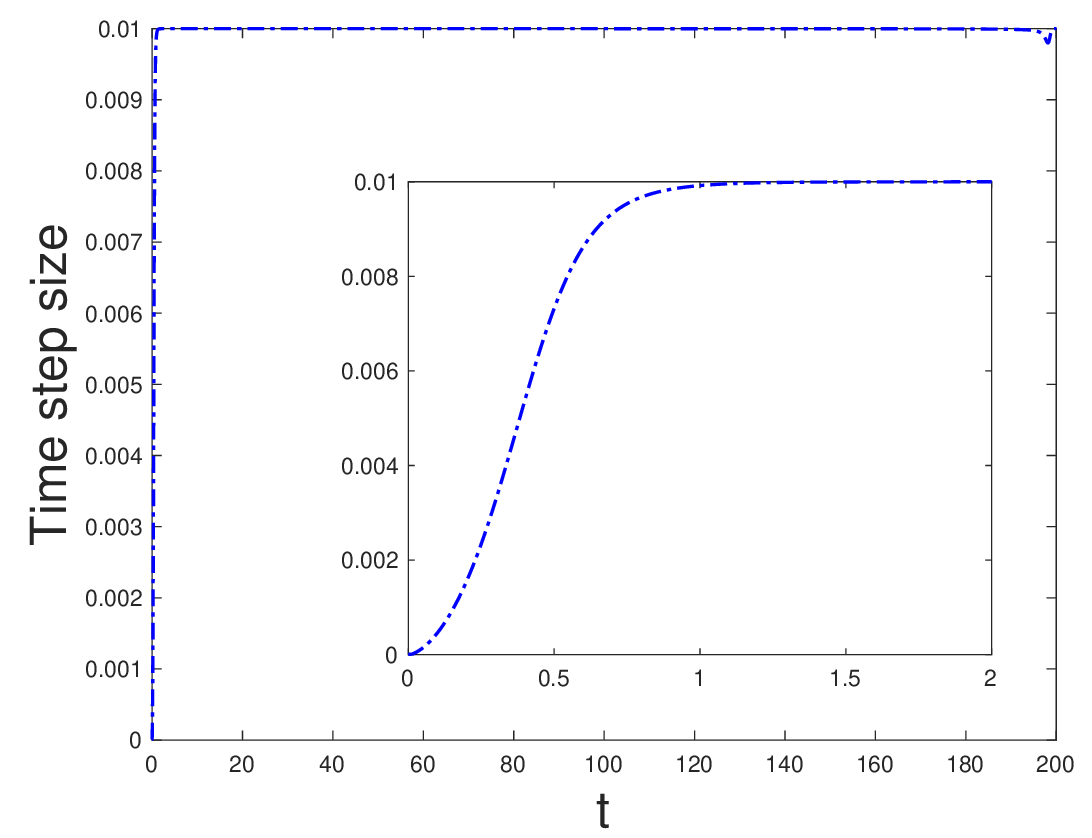}}
\centerline{(d) the time step size}
\end{minipage}
\caption{The evolutions in time of the radius, the supremum norm, and the energy of the simulated solution  and the time step size produced by the CN scheme \eqref{CN_2}  with the adaptive time stepping approach \eqref{adp} for the shrinking bubble problem.}\label{fig3}
\end{figure*}

\section {Concluding remarks}
A linear doubly stabilized CN scheme is constructed for the Allen-Cahn equation with general mobility in this paper.
The resulting fully-discrete system is formed by applying the central finite difference method for spatial discretization, and requires two Poisson-type equations  to be solved at each time step.
Two stabilizing terms are introduced to unconditionally preserve the discrete MBP of the proposed scheme. The discrete $H^{1}$ and $L^{\infty}$ error estimates are rigorously derived for the constant mobility case and the general one, respectively. Furthermore, the corresponding energy stability of the proposed scheme is also established for both cases.
Finally, a series of numerical experiments were carried out to verify the theoretical claims and illustrate the efficiency of the doubly stabilized CN scheme with a time adaptive strategy.
It remains interest to further theoretically explore the unconditional energy dissipation preservation of the proposed scheme, which has been numerically observed in our numerical experiment.

\section*{Acknowledgments}
The work of D. Hou is partially supported by the National Science Foundation of China grant 12001248,  the National Science Foundation of Jiangsu Province grant
BK20201020,  Jiangsu Province Universities  Science Foundation grant 20KJB110013 and  the Hong Kong Polytechnic University grant 1-W00D;
L. Ju's work is partially supported by US National Science Foundation grant DMS-2109633;
Z. Qiao‘s work   is partially supported by the Hong Kong Research Grants Council RFS grant RFS2021-5S03 and GRF grant 15303121, the Hong Kong Polytechnic University internal grant 1-9BCT, and CAS AMSS-PolyU Joint Laboratory of Applied Mathematics.

\bibliographystyle{plain}
\bibliography{ref}
\end{document}